\documentclass[11pt]{amsart}
\usepackage[latin1]{inputenc}
\usepackage[english]{babel}
\usepackage{amsmath}
\usepackage{amsthm}
\usepackage{amssymb}
\usepackage{amsthm, amsfonts, mathrsfs, amsfonts}
\usepackage{latexsym}
\usepackage{textcomp}
\usepackage{enumerate}

\usepackage{xcolor}
\usepackage{hyperref}

\numberwithin{equation}{section}
\newtheorem{theorem}{Theorem}[section]
\newtheorem{lemma}[theorem]{Lemma}

\newtheorem{prop}[theorem]{Proposition}
\newtheorem{cor}[theorem]{Corollary}
\newtheorem{Th}{Theorem}
\newtheorem{Cor}[Th]{Corollary}
\newtheorem{defn}{Definition}

\newtheorem{corollary}[theorem]{Corollary}

\theoremstyle{definition}

\newtheorem{remark}[theorem]{Remark}

\newtheorem{notation}[theorem]{Notation}
\newtheorem{definition}[theorem]{Definition}
\newtheorem*{Rem}{Remark}

\renewcommand{\phi}{\varphi}

\newcommand{\One}{\operatorname{\mathbf{1}}}

\theoremstyle{definition}

\usepackage[all]{xy}
\usepackage{color}

\newcommand{\mD}{\mathcal{D}}

\newcommand{\C}{\mathcal{C}}
\newcommand{\D}{\bold{D}}
\newcommand{\E}{\mathcal{E}}
\newcommand{\F}{\mathcal{F}}

\newcommand{\N}{\mathcal{N}}
\newcommand{\X}{\mathcal{X}}
\newcommand{\Y}{\mathcal{Y}}
\newcommand{\T}{\mathcal{T}}

\newcommand{\W}{\bold{W}}

\newcommand{\m}{\mathcal}
\newcommand{\<}{\langle}
\renewcommand{\>}{\rangle}

\newcommand{\Aut}{\mathrm{Aut}}
\newcommand{\Hom}{\mathrm{Hom}}

\newcommand{\Mor}{\mathrm{Mor}}
\newcommand{\Out}{\mathrm{Out}}
\newcommand{\Inn}{\mathrm{Inn}}
\newcommand{\Syl}{\mathrm{Syl}}
\newcommand{\Iso}{\mathrm{Iso}}

\newcommand{\Ob}{\mathrm{Ob}}

\newcommand{\id}{\operatorname{id}}

\newcommand{\tDelta}{\tilde{\Delta}}
\newcommand{\tGamma}{\tilde{\Gamma}}
\newcommand{\tS}{\tilde{S}}
\newcommand{\tF}{\tilde{\F}}

\newcommand{\tH}{\tilde{H}}

\newcommand{\norm}{\mathrel{\unlhd}}

\def \L {\mathcal L}
\def \H {\mathcal H}

\def \ov {\overline}

\title{Kernels of localities}

\author[V. Grazian]{Valentina Grazian}
\address{Department of Mathematics and Applications, University of Milano - Bicocca, Via Roberto Cozzi 55, 20125 Milano, Italy}
\email{valentina.grazian@unimib.it}

\author[E.~Henke]{Ellen Henke}
\address{Institut f{\"u}r Algebra, Fakult{\"a}t Mathematik, Technische Universit{\"a}t Dresden, 01062 Dresden, Germany}
\email{ellen.henke@tu-dresden.de}
\thanks{The authors were partially supported by EPSRC grant no EP/R010048/1. They moreover would like to thank the Isaac Newton Institute for Mathematical Sciences, Cambridge, for support and hospitality during the programme \emph{Groups, representations and applications: new perspectives}, where some of the work on this paper was undertaken and supported by EPSRC grant no EP/R014604/1.}

\begin{document}
\maketitle

\begin{abstract}
We state a sufficient condition for a fusion system to be saturated. This is then used to investigate localities with kernels, i.e. localities which are (in a particular way) extensions of groups by localities. As an application of these results, we define and study certain products in fusion systems and localities, thus giving a new method to construct saturated subsystems of fusion systems.
\end{abstract}

\section{Introduction}

The problem of showing that a given fusion system is saturated arises in different contexts. For example, proving that certain subsystems of  fusion systems are  saturated is one of the major difficulties in  developing a theory of saturated fusion systems in analogy to the theory of finite groups. When studying extensions of fusion systems, it is also crucial to understand under which conditions such extensions are saturated. In this paper, we seek to take up both themes simultaneously in a systematic way. To formulate and study  extension problems, it is common to work with linking systems or, more generally, with transporter systems associated to fusion systems (see e.g. \cite{BCGLOExtensions,OliverVentura,OliverExtensions}). The equivalent framework of localities (introduced by Chermak \cite{Ch}) was recently used by Chermak and the second named author of this paper to construct saturated subsystems of fusion systems, thereby enriching the theory of fusion systems by some new concepts (cf. \cite[Theorem~C]{normal} and \cite{Henke:NK}).


\smallskip

In the present paper (apart from the appendix) we work with localities rather than transporter systems. Our first main result is however formulated purely in terms of fusion systems. It gives a sufficient condition for a fusion system to be saturated. The proof generalizes an argument used by Oliver \cite{OliverExtensions} to show that the fusion systems associated to certain extensions of groups by linking systems are saturated.

\smallskip

Our saturation criterion serves us as an important tool for studying \emph{kernels} of  localities as introduced in Definition~\ref{D:Kernel} below. A kernel of a locality $\L$ is basically a partial normal  subgroup $\N$ such that the factor locality $\L/\N$ is a group and $\N$ itself supports the structure of a locality. We show in an appendix that kernels of localities correspond to ``normal pairs of transporter systems'' (cf. Definition~\ref{D:NormalPair}).

\smallskip

In Section~\ref{S:Products} we prove some results which demonstrate  that the theory of kernels can be used to construct saturated subsystems of fusion systems. More precisely, we  study certain products in localities which give rise to ``sublocalities'' whose fusion systems are saturated. As a special case, if $\F$ is a saturated fusion system over $S$, one can define a notion of a product of a normal subsystem with a subgroup of a model for $N_\F(S)$ (or equivalently with a saturated subsystem of $N_\F(S)$). In particular, our work generalizes the notion of a product of a normal subsystem with a subgroup of $S$, which was introduced by Aschbacher \cite{AschbacherGeneralized}. Our results on products are also used by the second named author of this article \cite{Henke:Subnormal} to construct normalizers and centralizers of subnormal subsystems of saturated fusion systems.

\smallskip

\textbf{For the remainder of this introduction let $\F$ be a fusion system over a finite $p$-group $S$.}

\smallskip

We will adapt the terminology and notation regarding fusion systems from \cite[Chapter~1]{AKO}, except that we will write homomorphisms on the right hand side of the argument (similarly as in \cite[Chapter~2]{AKO}) and that we will  define centric radical subgroups of $\F$ differently, namely as follows.

\begin{defn}\label{D:Fcr}
Define a subgroup $P\leq S$ to be \emph{centric radical} in $\F$ if
\begin{itemize}
 \item $P$ is centric, i.e. $C_S(Q)\leq Q$ for every $\F$-conjugate $Q$ of $P$; and
 \item $O_p(N_\F(Q))=Q$ for every fully $\F$-normalized $\F$-conjugate $Q$ of $P$.
\end{itemize}
Write $\F^{cr}$ for the set of subgroups of $S$ which are centric radical in $\F$.
\end{defn}

If $\F$ is saturated, then our notion of centric radical subgroups of $\F$ coincides with the usual notion (cf. Lemma~\ref{L:radical}). However, defining centric radical subgroups as above is crucial  if we want to conclude that the fusion systems of certain localities are saturated.

\subsection{A saturation criterion} The following definition will be used to formulate the previously mentioned sufficient condition for a fusion system to be saturated.

\begin{defn}\label{D:SaturationCriterion}
Let $\Delta$ be a set of subgroups of $S$.
\begin{itemize}
 \item The set $\Delta$ is called \emph{$\F$-closed} if $\Delta$ is closed under $\F$-conjugacy and overgroup-closed in $S$.
\item $\F$ is called $\Delta$-generated if every morphism in $\F$ can be written as a product of restrictions of $\F$-morphisms between subgroups in $\Delta$.
 \item $\F$ is called $\Delta$-saturated, if each $\F$-conjugacy class in $\Delta$ contains a subgroup which is fully automized and receptive in $\F$ (as defined in \cite[Definition~I.2.2]{AKO}).
\end{itemize}
\end{defn}

Generalizing arguments used by Oliver \cite{OliverExtensions} we prove the following theorem.

\begin{Th}\label{T:FInvariantOliverExtensions2}
Let $\E$ be an $\F$-invariant saturated subsystem of $\F$. Suppose $\F$ is $\Delta$-generated and $\Delta$-saturated for some $\F$-closed set $\Delta$ of subgroups of $S$ with $\E^{cr}\subseteq\Delta$. Then $\F$ is saturated.
\end{Th}

\subsection{Kernels of localities}\label{SS:Kernels}

The reader is referred to Section~\ref{S:Localities} for an introduction to partial groups and localities.  We say that a locality $(\L,\Delta,S)$ is a locality \emph{over} $\F$ to indicate that $\F=\F_S(\L)$. The set $\Delta$ is called the \emph{object set} of $(\L,\Delta,S)$. It follows from the definition of a locality that the normalizer $N_\L(P)$ of any object $P\in\Delta$ is a subgroup of $\L$ and thus a finite group. We will use the following definition.

\begin{defn}~
\begin{itemize}
 \item A locality $(\L,\Delta,S)$ over $\F$ is called \emph{cr-complete} if $\F^{cr}\subseteq\Delta$.
 \item A finite group $G$ is said to be \emph{of characteristic $p$} if $C_G(O_p(G))\leq O_p(G)$. 
 \item A locality $(\L,\Delta,S)$ is called a \emph{linking locality} if it is cr-complete and $N_\L(P)$ is of characteristic $p$ for every $P\in\Delta$.
\end{itemize}
\end{defn}

The slightly non-standard notion of centric radical subgroups introduced in Definition~\ref{D:Fcr} ensures that the fusion system $\F_S(\L)$ of a cr-complete locality $(\L,\Delta,S)$ is saturated (cf. Proposition~
\ref{P:crCompleteSaturated}(c)). If $(\L,\Delta,S)$ is a cr-complete locality over $\F$, then it gives rise to a transporter system $\m{T}$ associated to $\F$ whose object set is $\Delta$ and thus contains the set $\F^{cr}$. It follows from \cite[Proposition~4.6]{OliverVentura} that the $p$-completed nerve of such a transporter system $\m{T}$ is homotopy equivalent to the $p$-completed nerve of a linking system associated to $\F$.  The results we present next are centered around the following concept.

\begin{defn}\label{D:Kernel}
 A \emph{kernel} of a locality $(\L,\Delta,S)$ is a partial normal subgroup $\N$ of $\L$ such that $P\cap \N\in\Delta$ for every $P\in\Delta$.
\end{defn}

We show in Appendix~\ref{A} that kernels of localities correspond to ``normal pairs of transporter systems''. In particular, the results presented below can be translated to results on transporter systems. The reader is referred to Definition~\ref{D:NormalPair}, Proposition~\ref{L:KernelsToNormalPairs},   Theorem~\ref{T:NormalPairsToKernels} and  Remark~\ref{R:KernelsToNormalPairs} for details.

\smallskip

If $\N$ is a kernel of a locality $(\L,\Delta,S)$ then, setting
\[T:=\N\cap S\mbox{ and }\Gamma:=\{P\cap\N\colon P\in\Delta\},\]
it is easy to see that $(\N,\Gamma,T)$ is a locality (cf. Lemma \ref{L:KernelBasic}). We say in this situation also that $(\N,\Gamma,T)$ is a kernel of $(\L,\Delta,S)$.

\smallskip

Suppose now that $(\N,\Gamma,T)$ is a kernel of $(\L,\Delta,S)$. Observe that $T$ is an element of $\Gamma\subseteq\Delta$ and so $N_\L(T)$ is a subgroup of $\L$. It follows therefore from \cite[Theorem~4.3(b), Corollary~4.5]{loc1} that $\L/\N\cong N_\L(T)/N_\N(T)$ is a group. Thus, $\L$ can be seen as an extension of the group $\L/\N$ by the locality $(\N,\Gamma,T)$.

\smallskip

If the kernel $(\N,\Gamma,T)$ is cr-complete, then the following theorem implies that $\F_S(\L)$ is saturated. Its proof uses Theorem~\ref{T:FInvariantOliverExtensions2}.

\begin{Th}\label{T:crComplete}
Let $(\N,\Gamma,T)$ be a kernel of a locality $(\L,\Delta,S)$. Then $(\L,\Delta,S)$ is cr-complete if and only if $(\N,\Gamma,T)$ is cr-complete. If so, then $\F_T(\N)$ is a normal subsystem of $\F_S(\L)$.
\end{Th}

\begin{Th}\label{T:KernelLinkingLocality}
Let $(\L,\Delta,S)$ be a locality with a kernel $(\N,\Gamma,T)$. Then the following conditions are equivalent:
\begin{itemize}
 \item [(i)] $(\L,\Delta,S)$ is a linking locality;
 \item [(ii)] $(\N,\Gamma,T)$ is a linking locality and $N_\L(T)$ is of characteristic $p$;
 \item [(iii)] $(\N,\Gamma,T)$ is a linking locality and $C_\L(T)$ is of characteristic $p$.
\end{itemize}
\end{Th}

We now want to consider special kinds of linking localities. The object set of any linking locality over $\F$ is always contained in the set $\F^s$ of $\F$-subcentric subgroups (defined in  Definition~\ref{D:Subcentric}). If $\F$ is saturated, then the existence and uniqueness of centric linking systems implies conversely that, for every $\F$-closed set $\Delta$ of subgroups of $\F^s$, there is an essentially unique linking locality over $\F$ with object set $\Delta$.  Chermak introduced an $\F$-closed set $\delta(\F)\subseteq\F^s$, which by \cite[Lemma~7.21]{normal} can be described as the set of all subgroups of $S$ containing an element of $F^*(\F)^s$ (where $F^*(\F)$ is the generalized Fitting subsystem of $\F$ introduced by Aschbacher \cite{AschbacherGeneralized}). Notice that there exists always an essentially unique linking locality over $\F$ whose object set is the set $\delta(\F)$. Such a linking locality is called a \emph{regular locality}.

\smallskip

For an arbitrary locality $(\L,\Delta,S)$, there is a largest subgroup $R$ of $S$ with $\L=N_\L(R)$. This subgroup is denoted by $O_p(\L)$. Setting $\tDelta:=\{P\leq S\colon PO_p(\L)\in\Delta\}$, the triple $(\L,\tDelta,S)$ is also a locality (but with a possibly larger object set). It turns out that $(\L,\Delta,S)$ is a linking locality if and only if $(\L,\tDelta,S)$ is a linking locality (cf. \cite[Lemma~3.28]{normal}). This  flexibility in the choice of object sets makes it possible to formulate a result similar to Theorem~\ref{T:KernelLinkingLocality} for regular localities.

\begin{Th}\label{T:Regular}
Let $(\L,\Delta,S)$ be a locality with a kernel $(\N,\Gamma,T)$. Set 
\[\tDelta:=\{P\leq S\colon PO_p(\L)\in\Delta\}\mbox{ and }\tGamma:=\{Q\leq T\colon QO_p(\N)\in\Gamma\}.\] Then the following conditions are equivalent:
\begin{itemize}
 \item[(i)] $(\L,\tDelta,S)$ is a regular locality;
 \item[(ii)] $(\N,\tGamma,T)$ is a regular locality and $N_\L(T)$ is of characteristic $p$;
 \item[(iii)] $(\N,\tGamma,T)$ is a regular locality and $C_\L(T)$ is of characteristic $p$.
\end{itemize}
Moreover, if these conditions hold, then $E(\L)=E(\N)$.
\end{Th}

In an unpublished preprint, Chermak defined a locality $(\L,\Delta,S)$ to be \emph{semiregular} if (in our language) it has a kernel $(\N,\Gamma,T)$ which is a  regular locality. He  observed furthermore that a locality is semiregular if and only if it is an image of a regular locality under a projection of localities. As a consequence, images of semiregular localities under projections are semiregular. Moreover, since partial normal subgroups of regular localities form regular localities, it follows that partial normal subgroups of semiregular localities form semiregular localities. Thus, the category of semiregular localities and projections might provide a good framework to study extensions. This is one of our motivations to study kernels of localities more generally.

\begin{Rem}
Extensions of partial groups and localities have already been studied by Gonzalez \cite{Gonzalez}. He starts by giving important insights into the existence of extensions of partial groups. Basically, Gonzalez considers partial groups as simplicial sets and uses the concept of a simplicial fibre bundle. Gonzalez then proceeds to stating some results about extensions of localities in Section 7 of his paper. He calls a locality ``saturated'' if it is cr-complete in our sense. Under certain conditions it is shown that extensions of localities lead to (saturated) localities. To summarize, Gonzalez starts by defining \emph{isotypical extensions} (cf. \cite[Definition~7.1]{Gonzalez}) and shows then that an isotypical extension of a locality $(\mathtt{L}'',\Delta'',S'')$ by a locality $(\mathtt{L}',\Delta',S')$ leads to a locality $(\mathtt{T},\Delta,S)$ (cf. \cite[Proposition~7.6]{Gonzalez}). Slightly more precisely, we have $\mathtt{T}\subseteq\mathtt{L}$ for an extension $\mathtt{L}$ of the partial group $\mathtt{L}''$ by the partial group $\mathtt{L}'$.

\smallskip

The situation Gonzalez studies is principally different from ours. However, in \cite[Example~7.9, Corollary~7.10]{Gonzalez}, he considers a setup where $\mathtt{L}=\mathtt{T}$ (with $\mathtt{L}$ and $\mathtt{T}$ as above). In this situation one can observe easily that $(\mathtt{L}',\Delta',S')$ is a kernel of $(\mathtt{L},\Delta,S)$. Indeed, our Theorem~\ref{T:crComplete} shows that the assumption in \cite[Corollary~7.10]{Gonzalez} that $\Delta'$ contains all $\F'$-centric subgroups is redundant. It would be the subject of further research to see in how far our results have other interesting applications in the context of Gonzalez' work.
\end{Rem}


\subsection{Products in regular localities and in fusion systems} We now demonstrate that the theory of kernels can be used to study certain products in regular localities and thereby construct saturated subsystems of saturated fusion systems.

\smallskip

We study regular localities rather than arbitrary (linking) localities, mainly because every partial normal subgroup of a regular locality can be given the structure of a regular locality. To be exact, if $(\L,\Delta,S)$ is a regular locality and $\N\unlhd\L$, then $\E:=\F_{S\cap\N}(\N)$ is saturated and $(\N,\delta(\E),S\cap\N)$ is a regular locality.

\smallskip

In localities or, more generally, in partial groups, there is a natural notion of products of subsets. More precisely, if $\L$ is a partial group with product $\Pi\colon \D\rightarrow\L$, then
for $\X,\Y\subseteq\L$ we set
\[\X\Y:=\{\Pi(x,y)\colon x\in\X,\;y\in\Y,\;(x,y)\in\D\}.\]
The product of a partial normal subgroup with another partial subgroup is only in special cases known to be a partial subgroup. For example, the product of two partial normal subgroups of a locality $(\L,\Delta,S)$ is a partial normal subgroup (and forms thus a regular locality if $(\L,\Delta,S)$ is regular). Our next theorem gives a further example of a product which is a partial subgroup and can be given the structure of a regular locality.

\smallskip

For the theorem below to be comprehensible, a few preliminary remarks may be useful. For any linking locality $(\L,\Delta,S)$, one can define the generalized Fitting subgroup $F^*(\L)$ as a certain partial normal subgroup of $\L$ (see \cite[Definition~3]{Henke:Regular}). If $(\L,\Delta,S)$ is a regular locality, then $F^*(\L)$ is a kernel of $(\L,\Delta,S)$ and thus $T^*:=S\cap F^*(\L)$ is an element of $\Delta$. In particular, $N_\L(T^*)$ forms a group of characteristic $p$. 
If $\N\unlhd \L$ is a partial normal subgroup of $\L$, then $N_\N(T^*)$ is a normal subgroup of $N_\L(T^*)$. Hence, for any subgroup $H$ of $N_\L(T^*)$, the product $N_\N(T^*)H$ is a subgroup of $N_\L(T^*)$.

\begin{Th}\label{T:ProductMain}
Let $(\L,\Delta,S)$ be a regular locality. Fix moreover
\[\N\unlhd\L,\; T:=\N\cap S,\;\E:=\F_T(\N),\;T^*:=F^*(\L)\cap S\mbox{ and }H\leq N_\L(T^*).\]
Then $\N H$ is a partial subgroup of $\L$. Moreover, for every Sylow $p$-subgroup $S_0$ of $N_\N(T^*)H$ with $T\leq S_0$ the following hold:
\begin{itemize}
 \item [(a)] There exists a unique set $\Delta_0$ of subgroups of $S_0$ such that $(\N H,\Delta_0,S_0)$ is a cr-complete locality with kernel $(\N,\delta(\E),T)$.
 \item [(b)] Let $\Delta_0$ be as in (a) and set $\tDelta_0:=\{P\leq S\colon PO_p(\N H)\in\Delta_0\}$. Then $(\N H,\tDelta_0,S_0)$ is a regular locality  if and only if $N_\N(T^*)H$ is a group of characteristic $p$.
\end{itemize}
\end{Th}

If the hypothesis of Theorem~\ref{T:ProductMain} holds and $S\cap H$ is a Sylow $p$-subgroup of $H$, then
\[S_0:=T(S\cap H)\]
is a Sylow $p$-subgroup of $N_\N(T^*)H$ which is contained in $S$ (cf. Lemma~\ref{P:ChooseS0}). Thus, the cr-complete locality $(\N H,\Delta_0,S_0)$ from Theorem~\ref{T:ProductMain}(a) gives rise to a saturated subsystem $\F_{S_0}(\N H)=\F_{T(H\cap S)}(\N H)$ of $\F$. This leads us to a statement that can be formulated purely in terms of fusion systems. We use here that, for every regular locality $(\L,\Delta,S)$ over $\F$, there is by \cite[Theorem~A]{normal} a bijection from the set of partial normal subgroups of $\L$ to the set of normal subsystems of $\F$ given by $\N\mapsto\F_{S\cap\N}(\N)$. By \cite[Theorem~E(d)]{normal}, this bijection takes $F^*(\L)$ to $F^*(\F)$. In particular, $F^*(\F)$ is a fusion system over $F^*(\L)\cap S$.

\smallskip

To formulate the result we obtain, we rely on the fact that every constrained fusion system is realized by a model, i.e. by a finite group of characteristic $p$. Furthermore, if $\F$ is constrained and $G$ is a model for $\F$, then every normal subsystem of  $\F$ is realized by a normal subgroup of $G$. We also use that, for every saturated fusion system $\F$ over $S$, every normal subsystem $\E$ of $\F$ and every subgroup $R$ of $S$, there is a product subsystem $\E R$ defined (cf. \cite[Chapter~8]{AschbacherGeneralized} or \cite{Henke.products}).

\begin{Cor}\label{C:ProductsFusionSystems}
 Let $\F$ be a saturated fusion system over $S$ and $\E\unlhd \F$ over $T\leq S$. Let $T^*,T_0\leq S$ such that
 \[F^*(\F)\mbox{ is a fusion system over }T^*\mbox{ and $E(\E)$ is a fusion system over $T_0$}.\]
 Then $N_\F(T^*)$ is constrained and $N_\E(T_0)\unlhd N_\F(T^*)$. Thus we may choose a model $G$ for $N_\F(T^*)$ and $N\unlhd G$ with $\F_{S\cap N}(N)=N_\E(T_0)$. Let $H\leq G$ with $S\cap H\in\Syl_p(H)$. Set
\[S_0:=T(S\cap H)\mbox{ and }\E H:=\<E(\E)S_0,\F_{S_0}(NH)\>.\]
Then the following hold:
\begin{itemize}
 \item [(a)] $\E H$ is a saturated fusion system over $S_0$ with $\E\unlhd \E H$.
 \item [(b)] If $\mD$ is a saturated subsystem of $\F$ with $E(\E)\unlhd \mD$ and $\F_{S_0}(N H)\subseteq\mD$, then $\E H\subseteq \mD$.
\end{itemize}
\end{Cor}

For a saturated fusion system $\F$ with $\E\unlhd\F$ it is shown e.g. in  \cite[Lemma~7.13(a)]{normal} that $E(\E)\unlhd\F$. Hence, it makes indeed sense in the situation above to form the product subsystem $E(\E)S_0$. Moreover, part (b) of Corollary~\ref{C:ProductsFusionSystems} implies the following statement: If $\mD$ is a saturated subsystem of $\F$ with $\E\unlhd\mD$ and $\F_{S_0}(N H)\subseteq\mD$, then $\E H\subseteq \mD$.

\subsubsection*{Organization of the paper} We start by proving our saturation criterion (Theorem~\ref{T:FInvariantOliverExtensions2}) in Section~\ref{S:Saturation}. An introduction to partial groups and localities is then given in Section~\ref{S:Localities}. After proving some preliminary results, we study kernels of localities in Section~\ref{S:Kernels}. This is used in Section~\ref{S:Products} to prove Theorem~\ref{T:ProductMain} and Corollary~\ref{C:ProductsFusionSystems} as well as some more detailed results on products.

\section{Proving Saturation}\label{S:Saturation}

\textbf{Throughout this section let $\F$ be a fusion system over $S$.} 

\smallskip

In this section we prove Theorem~\ref{T:FInvariantOliverExtensions2}. The reader is referred to \cite[Chapter~I]{AKO} for an introduction to fusion systems. We will adopt the notation and terminology from there with the following two caveats: firstly, we write homomorphisms on the right hand side of the arguments similarly as in \cite[Chapter~II]{AKO}. Secondly, we define the set $\F^r$ of $\F$-radical subgroups differently, namely as in Definition~\ref{D:FcrCritical} below.

\begin{definition}
A subgroup $P\leq S$ is said to respect $\F$-saturation if there exists an element of $P^\F$ which is fully automized and receptive in $\F$ (as defined in \cite[Definition~I.2.2]{AKO}).
\end{definition}

If $\Delta$ is a set of subgroups of $S$ which is closed under $\F$-conjugacy, then notice that $\F$ is $\Delta$-saturated (as defined in the introduction) if and only if every element of $\Delta$ respects $\F$-saturation. On the other hand, $P$ respects $\F$-saturation if $\F$ is $P^\F$-saturated. Observe also that a fusion system is saturated if every subgroup of $S$ respects $\F$-saturation, or equivalently if $\F$ is $\Delta$-saturated, where $\Delta$ is the set of all subgroups of $S$. Roberts and Shpectorov \cite{RS} proved the following lemma, which we will use from now on, most of the time without reference.

\begin{lemma}\label{L:SylowExtensionAxiom}
Let $\m{C}$ be an $\F$-conjugacy class of $\F$. Then $\F$ is $\m{C}$-saturated if and only if the following two conditions hold:
\begin{itemize}
 \item [(I)] \emph{(Sylow axiom)} Each subgroup $P\in\m{C}$ which is fully $\F$-normalized is also fully $\F$-centralized and fully automized in $\F$.
 \item [(II)] \emph{(Extension axiom)} Each subgroup $P\in \m{C}$ which is fully $\F$-centralized is also receptive in $\F$.
\end{itemize}
Furthermore, if $\F$ is $\m{C}$-saturated, then for every fully $\F$-normalized $P\in\m{C}$ and every $Q\in P^\F$, there exists $\alpha\in\Hom_\F(N_S(Q),N_S(P))$ such that $Q\alpha=P$.
\end{lemma}

\begin{proof}
If (I) and (II) hold, then every fully $\F$-normalized subgroup $P\in\m{C}$ is fully automized and receptive, and thus $\F$ is $\m{C}$-saturated. On the other hand, if $\F$ is $\m{C}$-saturated, it follows from \cite[Lemma~I.2.6(c)]{AKO} that (I) and (II) and the statement of the lemma hold.
\end{proof}

\begin{cor}\label{C:FullyNormalizedConjugate}
If $\F$ is saturated and $P\leq S$ is fully $\F$-normalized, then for every $Q\in P^\F$ there exists $\alpha\in\Hom_\F(N_S(Q),N_S(P))$ such that $Q\alpha=P$.
\end{cor}

Recall that a subgroup $P\leq S$ is called \emph{$\F$-centric} if $C_S(Q)\leq Q$ for every $Q\in P^\F$. Write $\F^c$ for the set of $\F$-centric subgroups of $S$. We will need the following lemma.

\begin{lemma}\label{L:FcFqCharacterize}
 Let $\F$ be a fusion system over $S$ and let $P\leq S$ be fully $\F$-centralized. Then $P$ is $\F$-centric if and only if $C_S(P)\leq P$.
\end{lemma}

\begin{proof}
If $P$ is $\F$-centric, then clearly $C_S(P)\leq P$. Suppose now $C_S(P)\leq P$. Then for any $Q\in P^\F$, we have $C_S(Q)\geq Z(Q)\cong Z(P)=C_S(P)$. Hence, as $P$ is fully  $\F$-centralized, we have $C_S(Q)=Z(Q)\leq Q$ for all $Q\in P^\F$, i.e. $P$ is  $\F$-centric.
\end{proof}

We recall that $O_p(\F)$ denotes the largest subgroup of $S$ that is normal in $\F$ (\cite[Definition I.4.3]{AKO}). Moreover, if $\F$ is saturated and $O_p(\F) \in \F^c$, then $\F$ is called \emph{constrained} (\cite[Definition I.4.8]{AKO}). The Model Theorem for constrained fusion systems \cite[Theorem III.5.10]{AKO} guarantees that every constrained fusion system $\F$ over $S$ has a model, that is, there is a finite group $G$ such that  $S \in \Syl_p(G)$, $\F_S(G) = \F$ and  $C_G(O_p(G)) \leq O_p(G)$.

\begin{definition}\label{D:FcrCritical}~
\begin{itemize}
 \item A subgroup $P\leq S$ is called \emph{$\F$-radical} if there exists a fully $\F$-normalized $\F$-conjugate $Q$ of $P$ such that  $O_p(N_\F(Q)) = Q$. We denote by  $\F^r$ the set of $\F$-radical subgroups of $S$.
\item Set $\F^{cr} = \F^c \cap \F^r$ and call the elements of $\F^{cr}$ the $\F$-centric radical subgroups of $S$.
\item A subgroup $P\leq S$ is called \emph{$\F$-critical} if $P$ is  $\F$-centric and, for every $\F$-conjugate $Q$ of $P$, we have
\[\Out_S(Q)\cap O_p(\Out_\F(Q))=1.\] 
\end{itemize}
\end{definition}

As remarked before, our definition of radical subgroups differs from the usual one given for example in \cite[Definition~I.3.1]{AKO}. We show in part (b) of our next lemma that, for a saturated fusion system $\F$, the set $\F^{cr}$ equals the set of $\F$-centric radical subgroups in the usual definition. 

\begin{lemma}\label{L:radical}~
\begin{itemize}
\item [(a)] For every $R\leq S$ the following implications hold:
\[R\in\F^c\mbox{ and }O_p(\Aut_\F(R))=\Inn(R)\Longrightarrow R\mbox{ is $\F$-critical }\Longrightarrow R\in\F^{cr}.\]
\item [(b)] If $\F$ is saturated, then we have 
\[\F^{cr}=\{R\in\F^c\colon O_p(\Aut_\F(R))=\Inn(R)\}=\{R\leq S\colon R\mbox{ is $\F$-critical }\}.\] 
\end{itemize}
\end{lemma}

\begin{proof}
If $\Inn(R)=O_p(\Aut_\F(R))$, then $\Inn(Q)=O_p(\Aut_\F(Q))=\Aut_S(Q)\cap O_p(\Aut_\F(Q))$ for every $\F$-conjugate $Q$ of $R$, so $R$ is $\F$-critical if in addition $R\in\F^c$. This shows the first implication in (a).

\smallskip
 
Let now $R\in\F^c$ such that $R\not\in\F^r$. If we pick a fully  $\F$-normalized $\F$-conjugate $Q$ of $R$, we have then $Q<Q^*:=O_p(N_\F(Q))$. So  $\Inn(Q)<\Aut_{Q^*}(Q)$ as $Q\in\F^c$. Moreover, $\Aut_{Q^*}(Q)$ is normal in $\Aut_\F(Q)$, as by definition of $Q^*$ every element of $\Aut_\F(Q)$ extends to an element of $\Aut_\F(Q^*)$. Hence, $\Inn(Q)<\Aut_{Q^*}(Q)\leq \Aut_S(Q)\cap O_p(\Aut_\F(Q))$. This shows that $R$ is not $\F$-critical, so (a) holds. In particular 
\[ \F^{cr} \supseteq \{R\leq S\colon R\mbox{ is $\F$-critical }\} \supseteq \{R\in\F^c\colon O_p(\Aut_\F(R))=\Inn(R)\}.\]
For the proof of (b), it is thus sufficient to show that $O_p(\Aut_\F(R))=\Inn(R)$ for every $R\in\F^{cr}$. Fix now $R\in\F^{cr}$. Since the property $O_p(\Aut_\F(R))=\Inn(R)$ is preserved if $R$ is replaced by an $\F$-conjugate, we may assume without loss of generality that $R$ is fully  $\F$-normalized and $R=O_p(N_\F(R))$. Note that $N_\F(R)$ is saturated. So as $R\in\F^c$, the subsystem $N_\F(R)$ is constrained. Thus we may choose a model $G$ for $N_\F(R)$. Then $O_p(G)=R=O_p(N_\F(R))$ and
\[\Aut_\F(R)\cong G/C_G(R)\cong G/Z(R).\]
Hence, $O_p(\Aut_\F(R))\cong O_p(G/Z(R))=R/Z(R)\cong\Inn(R)$. Thus $O_p(\Aut_\F(R))= \Inn(R)$. This proves that $\F^{cr} \subseteq  \{R\in\F^c\colon O_p(\Aut_\F(R))=\Inn(R)\}\}$ and so (b) holds.
\end{proof}

The $\F$-critical subgroups play a crucial role in showing that a fusion system is saturated. This is made precise in the following theorem which we restate for the reader's convenience.

\begin{theorem}\label{T:BCGLOControlling}
Suppose $\Delta$ is a set of subgroups of $S$ which is closed under $\F$-conjugacy and contains every $\F$-critical subgroup. If $\F$ is $\Delta$-generated and $\Delta$-saturated, then $\F$ is saturated. 
\end{theorem}

\begin{proof}
This is a reformulation of \cite[Theorem~2.2]{controlling}. The reader might also want to note that the theorem follows from Lemma~\ref{L:BCGLOMainArgument} below.   
\end{proof}

Later on, we will need to prove saturation in a situation where it appears impossible to apply Theorem~\ref{T:BCGLOControlling} directly. We will therefore now have a closer look at the arguments used in the proof of that theorem. 

\smallskip

Define a partial order $\preceq$ on the set of $\F$-conjugacy classes by writing $\m{P}\preceq\m{Q}$ if some (and thus every) element of $\m{Q}$ contains an element of $\m{P}$.

\begin{lemma}\label{L:BCGLOLemma1}
Let $\H$ be a set of subgroups of $S$ closed under $\F$-conjugacy such that $\F$ is $\H$-generated and $\H$-saturated. Let $\m{P}$ be an $\F$-conjugacy class which is maximal with respect to $\preceq$ among those $\F$-conjugacy classes which are not contained in $\H$.

\smallskip

\noindent Write $\m{S}_{\geq P}\supseteq\m{S}_{>P}$ for the sets of subgroups of $N_S(P)$ which contain, or properly contain, an element $P\in\m{P}$. Then the following hold for every $P\in\m{P}$ which is fully  $\F$-normalized.
\begin{itemize}
 \item [(a)] The subsystem $N_\F(P)$ is $\m{S}_{>P}$-generated and $\m{S}_{>P}$-saturated.
 \item [(b)] If $N_\F(P)$ is $\m{S}_{\geq P}$-saturated, then $\F$ is $\m{H}\cup\m{P}$-saturated.
 \item [(c)] $P$ is fully $\F$-centralized. Moreover, for every $Q\in \m{P}$, there exists $\alpha\in\Hom_\F(N_S(Q),N_S(P))$ such that $Q\alpha=P$.
 \item [(d)] If $C_S(P)\leq P$, then $P$ is $\F$-centric, and if $\Out_S(P)\cap O_p(\Out_\F(P))=1$, then $\Out_S(Q)\cap O_p(\Out_\F(Q))=1$ for all $Q\in \m{P}$. In particular, if $C_S(P)\leq P$ and  $\Out_S(P)\cap O_p(\Out_\F(P))=1$, then $P$ is $\F$-critical.
\end{itemize}
\end{lemma}

\begin{proof}
 Basically, this follows from \cite[Lemma~2.4]{controlling} and its proof. We take the opportunity to point out the following small error in the proof of part (b) of that lemma: in l.9 on p.334 of \cite{controlling} it says ``replacing each $\phi_i$ by $\chi_i\circ\phi_i\circ\chi_i^{-1}\in\Hom_\F(\chi_i(Q_i),S)$''. However, it should be ``replacing each $\phi_i$ by $\chi_{i+1}\circ\phi_i\circ\chi_i^{-1}\in\Hom_\F(\chi_i(Q_i),S)$''.

\smallskip

Let us now explain in detail how the assertion follows. By \cite[Lemma~2.4(a),(c)]{controlling}, $N_\F(P)$ is $\m{S}_{>P}$-saturated and (b) holds. To prove (a), one needs to argue that $N_\F(P)$ is also $\m{S}_{>P}$-generated. If $\phi\in\Hom_{N_\F(P)}(A,B)$, then $\phi$ extends to a morphism  $\hat{\phi}\in\Hom_\F(AP,BP)$ which normalizes $P$. If $A\not\leq P$, then $AP\in\m{S}_{>P}$ and thus $\phi$ is the restriction of a morphism between subgroups in $\m{S}_{>P}$. If $A\leq P$, then $\hat{\phi}\in\Aut_\F(P)$, and by \cite[Lemma~2.4(b)]{controlling} $\hat{\phi}$ (and thus $\phi$) is a composite of restrictions of morphisms in $N_\F(P)$ between subgroups in $\m{S}_{>P}$. Hence $N_\F(P)$ is $\m{S}_{>P}$ generated and (a) holds. 

\smallskip

The following property is shown in the proof of Lemma~2.4 in \cite{controlling} (see p.333, property (3)).
\begin{itemize} 
\item [(*)] There is a subgroup $\hat{P}\in\m{P}$ fully  $\F$-centralized such that, for all $Q\in\m{P}$, there exists a morphism $\phi\in\Hom_\F(N_S(Q),N_S(\hat{P}))$ with $Q\phi=\hat{P}$. 
\end{itemize}
In particular, there exists $\psi\in\Hom_\F(N_S(P),N_S(\hat{P}))$ such that $P\psi=\hat{P}$. Since $P$ is fully  $\F$-normalized, the map $\psi$ is an isomorphism. In particular, $P$ is fully  $\F$-centralized as $\hat{P}$ is fully  $\F$-centralized. Moreover, if $Q\in\m{P}$, then by (*), there exists $\phi\in\Hom_\F(N_S(Q),N_S(\hat{P}))$ with $Q\phi=\hat{P}$, and we have then $\alpha:=\phi\psi^{-1}\in\Hom_\F(N_S(Q),N_S(P))$ with $Q\alpha=\hat{P}\psi^{-1}=P$. Hence, (c) holds.

\smallskip

It follows from Lemma~\ref{L:FcFqCharacterize} and the first part of (c) that $P$ is $\F$-centric if $C_S(P)\leq P$. If $Q\in\m{P}$, then the second part of (c) allows us to choose $\alpha\in\Hom_\F(N_S(Q),N_S(P))$ with $Q\alpha=P$. Notice that the map $\alpha^*\colon\Aut_\F(Q)\rightarrow \Aut_\F(P),\gamma\mapsto \alpha^{-1}\gamma\alpha$ is a group isomorphism with $\Inn(Q)\alpha^*=\Inn(P)$ and $\Aut_S(Q)\alpha^*\leq \Aut_S(P)$. So $\alpha^*$ induces an isomorphism from $\Out_\F(Q)$ to $\Out_\F(P)$ which maps $\Out_S(Q)\cap O_p(\Out_\F(Q))$ into $\Out_S(P)\cap O_p(\Out_\F(P))$. This implies (d).
\end{proof}

\begin{lemma}\label{L:BCGLOMainArgument}
Suppose that $\H$ is one of the following two sets:
\begin{itemize}
\item [(i)] the set of all subgroups of $S$ respecting $\F$-saturation; or 
\item [(ii)] the set of all $P\leq S$ such that every subgroup of $S$ containing an $\F$-conjugate of $P$ respects $\F$-saturation.  
\end{itemize}
Assume that $\F$ is $\H$-generated. Let $\m{P}$ be an $\F$-conjugacy class, which is maximal with respect to $\preceq$ among the $\F$-conjugacy classes not contained in $\H$. Then the elements of $\m{P}$ are $\F$-critical. 
\end{lemma}
 
\begin{proof}
Observe first that $\H$ is in either case closed under $\F$-conjugacy and $\F$ is $\H$-saturated. Moreover, note that $\F$ is not $\H\cup\m{P}$-saturated; this is clear if $\H$ is as in (i); if $\H$ is as in (ii) then note that, because of the maximal choice of $\m{P}$, $\F$ being $\H\cup\m{P}$-saturated would imply that every subgroup containing an element of $\m{P}$ would be either in $\m{P}$ or in $\H$ and thus respect $\F$-saturation.

\smallskip

Fix now $P\in\m{P}$ fully  $\F$-normalized. By Lemma~\ref{L:BCGLOLemma1}(a),(b) (using the notation introduced in that lemma), $N_\F(P)$ is $\m{S}_{>P}$ generated and $\m{S}_{>P}$-saturated, but not $\m{S}_{\geq P}$-saturated as $\F$ is not $\H\cup\m{P}$-saturated. Thus, by \cite[Lemma~2.5]{controlling} applied with $N_\F(P)$ in place of $\F$, we have $\Out_S(P)\cap O_p(\Out_\F(P))=1$ and $P\in N_\F(P)^c$. It follows now from Lemma~\ref{L:BCGLOLemma1}(d) that $P$ is $\F$-critical. 
\end{proof}

\begin{lemma}\label{L:FinvariantCritical}
Let $\E$ be an $\F$-invariant subsystem of $\F$ over $T\leq S$. Let $P\leq S$ and set $P_0:=P\cap T$. Then the following hold:
\begin{itemize}
 \item [(a)] If $C_S(P)\leq P$ and $\Out_S(P)\cap O_p(\Out_\F(P))=1$, then $C_T(P_0)\leq P_0$ and $\Out_T(P_0)\cap O_p(\Out_\E(P_0))=1$.
 \item [(b)] If $\E$ is saturated, $P_0$ is fully $\E$-normalized and $P$ is  $\F$-critical, then $P_0\in\E^{cr}$. 
\end{itemize}
\end{lemma}

\begin{proof}
Assume that $C_T(P_0)\not\leq P_0$ or $\Out_T(P_0)\cap O_p(\Out_\E(P_0))\neq 1$. Let $Q$ be the preimage of $\Out_T(P_0)\cap O_p(\Out_\E(P_0))$ in $N_T(P_0)$. Note that $C_T(P_0)\leq Q$, so our assumption implies in any case that $P_0<Q$ and thus $Q\not\leq P$. As $P$ normalizes $P_0$, $P$ also normalizes $Q$. Hence, $PQ$ is a $p$-group with $P<PQ$. This yields $P<N_{PQ}(P)=PN_Q(P)$ and thus $N_Q(P)\not\leq P$. Note that $[P,N_Q(P)]\leq P\cap T=P_0$. So $X:=\<\Aut_Q(P)^{\Aut_\F(P)}\>$ acts trivially on $P/P_0$. At the same time, by definition of $Q$, the elements of $\Aut_Q(P)$ and thus the elements of $X$ restrict to automorphisms of $P_0$ which lie in $O_p(\Out_\E(P_0))\unlhd \Aut_\F(P_0)$. Hence, a $p^\prime$-element of $X$ acts trivially on $P/P_0$ and on $P_0$, therefore it is trivial by properties of coprime action (cf. \cite[8.2.2(b)]{KS} or \cite[Lemma~A.2]{AKO}). This shows that $X$ is a $p$-group and so $\Aut_Q(P)\leq \Aut_S(P)\cap O_p(\Aut_\F(P))$. As $N_Q(P)\not\leq P$, this implies either $C_S(P)\not\leq P$ or $1\neq \Out_Q(P)\leq \Out_S(P)\cap O_p(\Out_\F(P))$. This proves (a).

\smallskip

Assume now that $\E$ is saturated, $P_0$ is fully $\E$-normalized, and $P$ is  $\F$-critical. The latter condition implies $C_S(P)\leq P$ and $\Out_S(P)\cap O_p(\Out_\F(P))=1$. So by (a), $C_T(P_0)\leq P_0$ and $\Out_T(P_0)\cap O_p(\Out_\E(P_0))=1$. As $\E$ is saturated and $P_0$ is fully $\E$-normalized, $P_0$ is fully $\E$-centralized and fully $\E$-automized. Hence, $P_0$ is $\E$-centric by Lemma~\ref{L:FcFqCharacterize} and $O_p(\Out_\E(P_0))\leq \Out_T(P_0)$. The latter condition implies $O_p(\Out_\E(P_0))=1$. Using Lemma~\ref{L:radical}(a) we conclude that $P_0\in\E^{cr}$. 
\end{proof}

The basic idea in the proof of the following theorem is taken from Step~5 of the proof of \cite[Theorem~9]{OliverExtensions}. As we argue afterwards, it easily implies Theorem~\ref{T:FInvariantOliverExtensions2}. If $P,Q\leq S$, $\phi\in\Hom_\F(P,Q)$ and $A\leq \Aut_\F(P)$, note that $A^\phi:=\phi^{-1}A\phi\leq \Aut_\F(P\phi)$.

\begin{theorem}\label{T:FInvariantOliverExtensions}
Let $\H$ be the set of all subgroups $P\leq S$ such that every subgroup of $S$ containing an $\F$-conjugate of $P$ respects $\F$-saturation. Assume that $\F$ is $\H$-generated. Suppose furthermore that there exists an $\F$-invariant saturated subsystem $\E$ of $\F$ with $\E^{cr}\subseteq\H$. Then $\F$ is saturated.
\end{theorem}

\begin{proof}
Let $T\leq S$ be such that $\E$ is a subsystem of $\F$ over $T$. By $\H_0$ denote the set of subgroups of $T$ which are elements of $\H$. Write $\H^\perp$ for the set of subgroups of $S$ not in $\H$, and $\H_0^\perp$ for the set of elements of $\H^\perp$ which are subgroups of $T$. If $P\leq S$, write $P_0$ for $P\cap T$, and similarly for subgroups of $S$ with different names.

\smallskip

For the proof of the assertion we will proceed in three steps. In the first two steps we will argue that property (\ref{*}) below implies the assertion, and in the third step we will prove that (\ref{*}) holds.
\begin{eqnarray}\label{*}
&&\mbox{Let }P\leq S\mbox{ such that }P_0\mbox{ is an element of }\H_0^\perp\mbox{ of maximal order. If }P_0\mbox{ is not fully}\\
&&\mbox{$\E$-normalized, then there exists $P''\in P^\F$ such that }|N_T(P_0)|<|N_T(P_0'')|.\nonumber 
\end{eqnarray}

\smallskip

\noindent\emph{Step~1:} We argue that (\ref{*}) implies the following property:
\begin{eqnarray}\label{**}
&&\mbox{If $P\leq S$ such that $P_0$ is an element of $\H_0^\perp$ of maximal order, then there exists an}\\
&&\mbox{$\F$-conjugate $Q$ of $P$ such that $Q_0$ is fully $\E$-normalized.}\nonumber
\end{eqnarray}
This can be seen as follows. If $P$ is as in (\ref{**}) and $P_0$ is fully $\E$-normalized, then we are done. If not, then by (\ref{*}), we can pass on to an $\F$-conjugate $P''$ of $P$ such that $|N_T(P_0)|<|N_T(P_0'')|$. Again, if $P_0''$ is fully $\E$-normalized, then we are done, otherwise we can repeat the process. Since $T$ is finite, we will eventually end up with an $\F$-conjugate $Q$ of $P$ such that $Q_0$ is fully $\E$-normalized. So (\ref{*}) implies (\ref{**}).

\smallskip

\noindent\emph{Step~2:} We show that property (\ref{**}) implies the assertion. Assume that (\ref{**}) holds and $\F$ is not saturated. Then $\H^\perp\neq\emptyset$. By construction, $\H\supseteq\H_0$ is $\F$-closed. So, for every subgroup $P\leq S$, $P_0\in\H_0$ implies $P\in\H$. In particular, as $\H^\perp\neq\emptyset$, we can conclude $\H_0^\perp\neq \emptyset$. Set 
\begin{eqnarray*}
 m&:=&\max\{|Q|\colon Q\in\H_0^\perp\},\\
\m{M}&:=&\{P\in\H^\perp\colon |P_0|=m\}.
\end{eqnarray*}
For any $Q\in\H_0^\perp$ with $|Q|=m$, we have $Q=Q_0$ and thus $Q\in\m{M}$. Hence, $\m{M}\neq\emptyset$. As $T$ is strongly closed, $\m{M}$ is closed under $\F$-conjugacy. Thus, we can choose an $\F$-conjugacy class $\m{P}\subseteq\m{M}$ such that the elements of $\m{P}$ are of maximal order among the elements of $\m{M}$.

\smallskip

We argue first that $\m{P}$ is maximal with respect to $\preceq$ among the $\F$-conjugacy classes contained in $\H^\perp$. For that, let $P\in\m{P}$ and $P\leq R\leq S$. We need to show that $R=P$ or $R\in\H$. Notice that $m=|P_0|\leq |R_0|$. If $m<|R_0|$, then by definition of $m$ we have $R_0\in\H_0\subseteq\H$ and thus $R\in\H$. If $m=|R_0|$, then $R\in\m{M}$ or $R\in\H$. So the maximality of $|P|$ yields $R=P$ or $R\in\H$. Hence, $\m{P}$ is maximal with respect to $\preceq$ among the $\F$-conjugacy classes contained in $\H^\perp$. Thus Lemma~\ref{L:BCGLOMainArgument} implies that the elements of $\m{P}$ are $\F$-critical. By (\ref{**}), there exists $P\in\m{P}$ such that $P_0$ is fully $\E$-normalized. Then by Lemma~\ref{L:FinvariantCritical}(b), we have $P_0\in\E^{cr}$. Since by assumption $\E^{cr}\subseteq\H$, it follows that $P_0\in\H$ and so $P\in\H$. This contradicts the choice of $\m{P}\subseteq\m{M}\subseteq\H^\perp$. Hence, (\ref{**}) implies that $\F$ is saturated. 

\smallskip

\noindent\emph{Step~3:} We complete the proof by proving (\ref{*}). Whenever we have subgroups $Q_1\leq Q\leq S$, we set $\Aut_\F(Q:Q_1):=N_{\Aut_\F(Q)}(Q_1)$, $\Aut_S(Q:Q_1):=N_{\Aut_S(Q)}(Q_1)$, $N_S(Q:Q_1):=N_S(Q)\cap N_S(Q_1)$ et cetera. Let now $P\leq S$ be such that $P_0$ is an element of $\H_0^\perp$ of maximal order which is not fully $\E$-normalized. Note that $T\in\E^{cr}\subseteq\H$ and thus $P_0\neq T$. Hence, $P_0<N_T(P_0)$. So the maximality of $|P_0|$ yields $R:=N_T(P_0)\in\H_0$. In particular, every $\F$-conjugate of $R$ respects $\F$-saturation. Since $\E$ is saturated, there exists $\rho\in\Hom_\E(R,T)$ such that $P_0':=P_0\rho$ is fully $\E$-normalized. As $R':=R\rho$ respects $\F$-saturation, there exists $\sigma\in\Hom_\F(R',S)$ such that $R'':=R'\sigma$ is fully automized and receptive. Set $P_0'':=P_0' \sigma=P_0 \rho\sigma $. As $R''$ is fully automized, by Sylow's Theorem there exists $\xi\in\Aut_\F(R'')$ such that
\[\Aut_\F(R'':P_0'')^\xi\cap\Aut_S(R'')=\Aut_S(R'': P_0'' \xi)\]
is a Sylow $p$-subgroup of $\Aut_\F(R'':P_0'')^\xi=\Aut_\F(R'': P_0'' \xi)$. So replacing $\sigma$ by $\sigma\xi$, we may assume
\[\Aut_S(R'':P_0'')\in\Syl_p(\Aut_\F(R'':P_0'')).\]
Notice that $\Aut_S(R:P_0)^{\rho\sigma}$ and $\Aut_S(R':P_0')^\sigma$ are $p$-subgroups of $\Aut_\F(R'':P_0'')$. Hence, again by Sylow's Theorem, there exist $\gamma,\delta\in\Aut_\F(R'':P_0'')$ such that $\Aut_S(R:P_0)^{\rho\sigma\gamma}$ and $\Aut_S(R':P_0')^{\sigma\delta}$ are contained in $\Aut_S(R'':P_0'')$. This implies $N_S(R:P_0)\leq N_{\rho\sigma\gamma}$ and $N_S(R':P_0')\leq N_{\sigma\delta}$. Hence, as $R''$ is receptive, the map $\rho\sigma\gamma\in \Hom_\F(R,R'')$ extends to $\phi\in\Hom_\F(N_S(R:P_0),S)$, and similarly $\sigma\delta\in\Hom_\F(R',R'')$ extends to $\psi\in\Hom_\F(N_S(R':P_0'),S)$. Notice that
\[P_0 \phi=P_0 {\rho\sigma\gamma}=P_0'' \gamma=P_0''\mbox{ and } P_0' \psi= P_0' {\sigma\delta}=P_0'' \delta=P_0''.\]
Moreover, since $P_0=P\cap T$, $R=N_T(P_0)$ and $T$ is strongly closed, we have $P\leq N_S(P_0)=N_S(R:P_0)$. Hence
\[P'':=P\phi\]
is a well defined $\F$-conjugate of $P$, for which we have indeed that $P''\cap T=P \phi\cap T=(P\cap T)\phi=P_0 \phi=P_0''$ as our notation suggests. So it remains only to show that $|N_T(P_0)|<|N_T(P_0'')|$. We will show this by arguing that
\begin{eqnarray}\label{10}
|N_T(P_0)|=|R|<|N_T(R':P_0')|\leq |N_T(P_0'')|.
\end{eqnarray}
Recall that $P_0$ is by assumption not fully $\E$-normalized, whereas $P_0'=P_0 \rho$ is fully $\E$-normalized. Therefore, we have $R'=R \rho=N_T(P_0) \rho<N_T(P_0')$. Thus $R'<N_{N_T(P_0')}(R')=N_T(R':P_0')$. As $|R|=|R'|$, this shows the first inequality in (\ref{10}). As $P_0' \psi=P_0''$ and $T$ is strongly closed, we have $N_T(R':P_0') \psi\leq N_T(P_0'')$. Since $\psi$ is injective, this shows the second inequality in (\ref{10}). So the proof of (\ref{*}) is complete. As argued in Steps~1 and 2, this proves the assertion.
\end{proof}

\begin{proof}[Proof of Theorem~\ref{T:FInvariantOliverExtensions2}]
 Write $\H$ for the set of all subgroups $P\leq S$ such that every subgroup of $S$ containing an $\F$-conjugate of $P$ respects $\F$-saturation. Since $\Delta$ is $\F$-closed and $\F$ is $\Delta$-saturated, it follows that $\Delta\subseteq\H$. Hence, as $\F$ is $\Delta$-generated, $\F$ is also $\H$-generated. Moreover, $\E^{cr}\subseteq\Delta\subseteq\H$. Thus, the assertion follows from Theorem~\ref{T:FInvariantOliverExtensions} above.
\end{proof}

\section{Partial groups and localities}\label{S:Localities}

In this section we introduce some basic definitions and notations that will be used in the remainder of this paper.  We refer the reader to \cite{loc1} for a more comprehensive introduction to partial groups and localities.

\subsection{Partial groups}

Following the notation introduced in \cite{loc1}, we will write $\W(\L)$ for the set of words in a set $\L$. The elements of $\L$ will be identified with the words of length one, and $\emptyset$ denotes the empty word. The concatenation of words $u_1,u_2,\dots,u_k\in\W(\L)$ will be denoted $u_1\circ u_2\circ\cdots\circ u_k$.

\begin{definition}[{\cite[Definition~1.1]{loc1}}]\label{partial}
Suppose $\L$ is a non-empty set and $\D\subseteq\W(\L)$. Let $\Pi \colon  \D \longrightarrow \L$ be a map, and let $(-)^{-1} \colon \L \longrightarrow \L$ be an involutory bijection, which we extend to a map 
\[(-)^{-1} \colon \W(\L) \longrightarrow \W(\L), w = (g_1, \dots, g_k) \mapsto w^{-1} = (g_k^{-1}, \dots, g_1^{-1}).\]
Then $\L$ is called a \emph{partial group} with product $\Pi$ and inversion $(-)^{-1}$ if the following hold for all words $u,v,w\in\W(\L)$:
\begin{itemize}
\item  $\L \subseteq \D$ and
\[  u \circ v \in \D \Longrightarrow u,v \in \D.\]
(So in particular, $\emptyset\in\D$.)
\item $\Pi$ restricts to the identity map on $\L$.
\item $u \circ v \circ w \in \D \Longrightarrow u \circ (\Pi(v)) \circ w \in \D$, and $\Pi(u \circ v \circ w) = \Pi(u \circ (\Pi(v)) \circ w)$.
\item $w \in  \D \Longrightarrow  w^{-1} \circ w\in \D$ and $\Pi(w^{-1} \circ w) = \One$ where $\One:=\Pi(\emptyset)$.
\end{itemize}
\end{definition}

\textbf{For the remainder of this section let $\L$ always be a partial group with product $\Pi\colon\D\rightarrow\L$. As above set $\One:=\Pi(\emptyset)$.}

\smallskip

If $w=(f_1,\dots,f_n)\in\D$, then we write sometimes $f_1f_2\cdots f_n$ for $\Pi(f_1,\dots,f_n)$. In particular, if $(x,y)\in\D$, then $xy$ denotes the product $\Pi(x,y)$.

\begin{notation}~
 \begin{itemize}
 \item For every $f\in\L$, write $\D(f):=\{x\in\L\colon (f^{-1},x,f)\in\D\}$ for the set of all $x$ such that the conjugate $x^f:=\Pi(f^{-1},x,f)$ is defined. 
 \item By $c_f$ denote the conjugation map $c_f\colon \D(f)\rightarrow \L,x\mapsto x^f$.
 \item For $f\in\L$ and  $\X\subseteq\D(f)$, set $\X^f:=\{x^f\colon x\in\X\}$.
\item Given $\X\subseteq\L$, set
\[N_\L(\X):=\{f\in\L\colon \X\subseteq\D(f)\mbox{ and }\X^f=\X\}\]
and
\[C_\L(\X):=\{f\in\L\colon x\in\D(f)\mbox{ and }x^f=x\mbox{ for all }x\in\X\}.\]
\item For $\X,\Y\subseteq\L$ define $N_\Y(\X)=\Y\cap N_\L(\X)$ and $C_\Y(\X)=\Y\cap C_\L(\X)$.  
\end{itemize}
\end{notation}

We call $N_\Y(\X)$ the \emph{normalizer} of $\X$ in $\Y$ and $C_\Y(\X)$ the \emph{centralizer} of $\X$ in $\Y$. 

\begin{definition}\label{D:PartialSubgroup}~
\begin{itemize}
 \item A subset $\H\subseteq\L$ is called a \emph{partial subgroup} of $\L$ if $h^{-1}\in\H$ for all $h\in\H$, and moreover $\Pi(w)\in\H$ for all $w\in\D\cap\W(\H)$.
 \item A partial subgroup $\H$ of $\L$ is a called a \emph{subgroup} of $\L$ if $\W(\H)\subseteq\D(\L)$.
 \item By a \emph{$p$-subgroup} of $\L$ we mean a subgroup $S$ of $\L$ such that $|S|$ is a power of $p$. 
 \item Let $\N$ be a partial subgroup of $\L$. Then we call $\N$ a \emph{partial normal subgroup} of $\L$ (and write $\N\unlhd\L$) if $n^f\in\N$ for all $f\in\L$ and all $n\in\N\cap\D(f)$.
 \item A partial subgroup $\H$ of $\L$ is called a \emph{partial subnormal subgroup} of $\L$ 
if there exists a series $\H=\H_0\unlhd\H_1\unlhd\cdots\unlhd \H_k=\L$ of partial subgroups of $\L$.
\end{itemize}
\end{definition}

If $\H$ is a partial subgroup of $\L$, then notice that $\H$ is itself a partial group with product $\Pi|_{\W(\H)\cap\D}$. If $\H$ is a subgroup of $\L$, then $\H$ forms actually a group with binary product defined by $x\cdot y:=\Pi(x,y)$ for all $x,y\in\H$. In particular, every $p$-subgroup of $\L$ forms a $p$-group. The following definition will be crucial in the definition of a locality. 

\begin{definition}
Let $\L$ be a partial group and let $\Delta$ be a collection of subgroups of $\L$. Define $\D_\Delta$ to be the set of words $w=(g_1, \dots, g_k) \in \W(\L)$ such that there exist $P_0, \dots ,P_k \in \Delta$ with $P_{i-1} \subseteq \D(g_i)$ and $P_{i-1}^{g_i} = P_i$ for all $1 \leq  i \leq k$. If such $w$ and $P_0,\dots,P_k$ are given, then we say also that $w\in\D_\Delta$ via $P_0,P_1,\dots,P_k$, or just that $w\in\D_\Delta$ via $P_0$.
\end{definition}

\begin{lemma}\label{L:DDeltaDGamma}
 Let $\L$ be a partial group and $\N\unlhd\L$.
\begin{itemize}
 \item [(a)] If $f\in\L$ and $P\subseteq\D(f)$, then $(P\cap\N)^f=P^f\cap\N$.
 \item [(b)] Let $S$ be a subgroup of $\L$ and $\Gamma$ a set of subgroups of $\N\cap S$. Set 
\[\Delta:=\{P\leq S\colon P\cap \N\in\Gamma\}.\]
Then $\D_\Gamma=\D_\Delta$. 
\end{itemize}
\end{lemma}

\begin{proof}
\textbf{(a)} Notice that $(P\cap \N)^f\subseteq P^f\cap\N$ as $\N\unlhd\L$. Since $(c_f)^{-1}=c_{f^{-1}}$ by \cite[Lemma~1.6(c)]{loc1}, we have $(P^f)^{f^{-1}}=P$. So similarly, $(P^f\cap \N)^{f^{-1}}\subseteq P\cap \N$ and thus $P^f\cap \N\subseteq (P\cap \N)^f$. This shows $(P\cap \N)^f=P^f\cap \N$ as required. 

\smallskip

\textbf{(b)} As $\Gamma\subseteq\Delta$, we have $\D_\Gamma\subseteq \D_\Delta$. Conversely, if $w=(f_1,\dots,f_n)\in\D_\Delta$ via $P_0,P_1,\dots,P_n\in\Delta$, then (a) yields that $w\in \D_\Gamma$ via $P_0\cap\N,P_1\cap\N,\dots,P_n\cap \N$. Thus $\D_\Delta \subseteq \D_\Gamma$, implying $\D_\Gamma=\D_\Delta$. 
\end{proof}

\subsection{Localities}

\begin{definition}
Let $\L$ be a partial group and let $S$ be a $p$-subgroup of $\L$. For $f\in\L$ set
\[S_f:=\{x\in S\colon x\in\D(f)\mbox{ and }x^f\in S\}.\]
More generally, if $w=(f_1,\dots,f_k)\in\W(\L)$, then write $S_w$ for the set of $s\in S$ such that there exists a sequence $s=s_0,\dots,s_k$ of elements of $S$ with $s_{i-1}\in\D(f_i)$ and $s_{i-1}^{f_i}=s_i$ for $i=1,\dots,k$.
\end{definition}

\begin{definition}\label{locality}
Let $\L$ be a finite partial group with product $\Pi\colon\D\rightarrow\L$, let $S$ be a $p$-subgroup of $\L$ and let $\Delta$ be a non-empty set of subgroups of $S$. We say that $(\L, \Delta, S)$ is a \emph{locality} if the following hold:
\begin{enumerate}
\item $S$ is maximal with respect to inclusion among the $p$-subgroups of $\L$;
\item $\D = \D_\Delta$;
\item $\Delta$ is closed under passing to $\L$-conjugates and overgroups in $S$; i.e. $\Delta$ is overgroup-closed in $S$ and $P^f\in \Delta$ for all $P\in\Delta$ and $f\in \L$ with $P \subseteq S_f$.
\end{enumerate}
If $(\L,\Delta,S)$ is a locality, then $\Delta$ is also called the set of \emph{objects} of $(\L,\Delta,S)$.
\end{definition}

As argued in \cite[Remark~5.2]{subcentric}, Definition~\ref{locality} is equivalent to the definition of a locality given by Chermak \cite[Definition 2.7]{loc1}. We will use this fact throughout.

\smallskip

\textbf{For the remainder of this subsection let $(\L,\Delta,S)$ be a locality.}

\begin{lemma}[Important properties of localities]\label{L:LocalitiesProp}
The following hold:
\begin{itemize}
\item [(a)] $N_\L(P)$ is a subgroup of $\L$ for each $P\in\Delta$.
\item [(b)] Let $P\in\Delta$ and $g\in\L$ with $P\subseteq S_g$. Then $Q:=P^g\in\Delta$, $N_\L(P)\subseteq \D(g)$ and 
$$c_g\colon N_\L(P)\rightarrow N_\L(Q)$$
is an isomorphism of groups. 
\item [(c)] Let $w=(g_1,\dots,g_n)\in\D$ via $(X_0,\dots,X_n)$. Then 
$$c_{g_1}\circ \dots \circ c_{g_n}=c_{\Pi(w)}$$
 is a group isomorphism $N_\L(X_0)\rightarrow N_\L(X_n)$.
\item [(d)] $S_g\in\Delta$ for every $g\in\L$. In particular, $S_g$ is a subgroup of $S$. Moreover, $S_g^g=S_{g^{-1}}$ and $c_g\colon S_g\rightarrow S$ is an injective group homomorphism.
\item [(e)] For every $w\in\W(\L)$, $S_w$ is a subgroup of $S$. Moreover, $S_w\in\Delta$ if and only if $w\in\D$. 
\item [(f)] If $w\in\D$, then $S_w\leq S_{\Pi(w)}$. 
\end{itemize}
\end{lemma}

\begin{proof}
Properties (a),(b) and (c) correspond to the statements (a),(b) and (c) in \cite[Lemma~2.3]{loc1} except for the fact stated in (b) that $Q\in\Delta$, which is however clearly true if one uses the definition of a locality above. Property (d) holds by \cite[Proposition~2.5(a),(b)]{loc1} and property (e) is stated in \cite[Corollary~2.6]{loc1}. Property (f) follows from (e) and (c).
\end{proof}

We use from now on without further reference that $c_g\colon S_g\rightarrow S$ is an injective group homomorphism for all $g\in\L$. By $\F_S(\L)$ we denote the fusion system over $S$ generated by all maps of this form. More generally, we have the following definition.

\begin{definition}
Let $\H$ be a partial subgroup of $\L$. Write $\F_{S\cap \H}(\H)$ for the fusion system over $S\cap \H$ generated by the injective group homomorphisms of the form $c_h\colon S_h\cap \H\rightarrow S\cap\H$ with $h\in\H$. 
\end{definition}

\begin{lemma}\label{L:NDPNHP}
Let $\H$ be a partial subgroup of $\L$ and $P\in\Delta$ with $P\leq S\cap \H$. Set $\mD:=\F_{S\cap \H}(\H)$. Then the following hold:
\begin{itemize}
 \item [(a)] For every $\phi\in\Hom_\mD(P,S\cap\H)$ there exists $h\in\H$ with $P\leq S_h$ and $\phi=c_h$.
 \item [(b)] $N_\mD(P)=\F_{N_{S\cap\H}(P)}(N_\H(P))$. 
\end{itemize}
\end{lemma}

\begin{proof}
For part (a) see \cite[Lemma~3.11(b)]{Henke:Regular}. Set $T:=S\cap \H$. Clearly $\F_{N_T(P)}(N_\H(P))\subseteq N_\mD(P)$ and $N_\mD(P)$ is a fusion system over $N_T(P)$. If $\phi\in\Hom_{N_\mD(P)}(X,Y)$ where $X,Y\leq N_T(P)$ is a morphism in $N_\mD(P)$, then we may assume that $P\leq X\cap Y$ and $P\phi=P$. As $P\in\Delta$, we have $X,Y\in\Delta$. Hence, by (a), there exists $h\in\H$ such that $X\leq S_h$ and   $\phi=c_h|_X$. As $P\phi=P$, it follows $h\in N_\H(P)$ and thus $\phi$ is a morphism in $\F_{N_T(P)}(N_\H(P))$ as required.
\end{proof}

\begin{definition}
If $(\L,\Delta,S)$ is a locality, then set
\[O_p(\L):=\bigcap\{S_w\colon w\in\W(\L)\}\] 
\end{definition}

\begin{lemma}\label{L:OpL}
If $(\L,\Delta,S)$ is a locality, then $O_p(\L)$ is the unique largest $p$-subgroup of $\L$ which is a partial normal subgroup of $\L$. Moreover, a subgroup $P\leq S$ is a partial normal subgroup of $\L$ if and only if $N_\L(P)=\L$.
\end{lemma}

\begin{proof}
 This is proved as \cite[Lemma~3.13]{Henke:Regular} based on \cite[Lemma~2.13]{loc1}.
\end{proof}

We will use the characterization of $O_p(\L)$ given in Lemma~\ref{L:OpL} from now on most of the time without further reference.

\begin{lemma}\label{L:FrattiniSplitting}
 Let $(\L,\Delta,S)$ be a locality, $\N\unlhd\L$ and $T:=\N\cap S$. Then for every $g\in\L$ there exist $n\in\N$ and $f\in N_\L(T)$ such that $(n,f)\in\D$, $g=nf$ and $S_g=S_{(n,f)}$.
\end{lemma}

\begin{proof}
By the Frattini Lemma \cite[Corollary~3.11]{loc1}, there exist $n\in\N$ and $f\in\L$ such that $(n,f)\in\D$, $g=nf$ and $f$ is $\uparrow$-maximal (in the sense of \cite[Definition~3.6]{loc1}). By \cite[Proposition~3.9]{loc1} and then \cite[Lemma~3.1(a)]{loc1}, every $\uparrow$-maximal element is in $N_\L(T)$. Moreover, the Splitting Lemma \cite[Lemma~3.12]{loc1} gives $S_g=S_{(n,f)}$.
\end{proof}

\begin{definition}\label{D:RestrictionLocality}
Let $(\L^+,\Delta^+,S)$ be a locality with a partial product $\Pi^+\colon\D^+\longrightarrow\L^+$. Suppose that  $\emptyset\neq \Delta\subseteq \Delta^+$ such that $\Delta$ is $\F_S(\L^+)$-closed. Set 
\[\L^+|_{\Delta}:=\{f\in\L^+\colon S_f\in\Delta\}.\]
Note that $\D:=\D_{\Delta}\subseteq\D^+\cap\W(\L^+|_{\Delta})$ and, by Lemma~\ref{L:LocalitiesProp}(c), $\Pi^+(w)\in\L|_{\Delta}$ for all $w\in\D$. We call $\L:=\L^+|_{\Delta}$ together with the partial product $\Pi^+|_{\D}\colon\D\longrightarrow \L$ and the restriction of the inversion map on $\L^+$ to $\L$ the \emph{restriction} of $\L^+$ to $\Delta$.
\end{definition}

With the hypothesis and notation as in the preceding definition, it turns out that the restriction of $\L^+$ to $\Delta$ forms a partial group and the triple $(\L^+|_\Delta,\Delta,S)$ is a locality (see \cite[Lemma~2.21(a)]{Ch} and \cite[Lemma~2.23(a),(c)]{Henke:2020} for details). It might be worth pointing out that in the definition of the restriction, $S_f$ and $\D_\Delta$ are a priori formed inside of $\L^+$, but as argued in \cite[Lemma~2.23(b)]{Henke:2020}, it does not matter whether one forms $S_f$ and $\D_\Delta$ inside of $\L^+$ or inside of the partial group $\L^+|_\Delta$.

\subsection{Homomorphisms of partial groups and projections}

\begin{definition}\label{D:Homomorphisms}
Let $\L$ and $\L'$ be partial groups with products $\Pi\colon\D\rightarrow \L$ and $\Pi'\colon \D'\rightarrow \L'$ respectively. Let $\alpha\colon \L\rightarrow\L',f\mapsto f\alpha$ be a map. 
\begin{itemize}
\item Write $\alpha^*$ for the induced map 
\[\alpha^*\colon \W(\L)\rightarrow \W(\L'), (f_1,\dots,f_n)\mapsto(f_1\alpha,\dots,f_n\alpha).\] 
\item We call $\alpha$ a \emph{homomorphism of partial groups} if $\D\alpha^*\subseteq\D'$ and $\Pi(w)\alpha=\Pi'(w\alpha^*)$. 
\item The map $\alpha$ is called a \emph{projection of partial groups}, if $\alpha$ is a homomorphism of partial groups and $\D\alpha^*=\D'$. 
\item A bijective projection of partial groups is called an \emph{isomorphism} of partial groups, and an isomorphism from $\L$ to itself is called an \emph{automorphism} of $\L$. Write $\Aut(\L)$ for the set of automorphisms of $\L$. 
\item If $\alpha\colon\L\rightarrow\L'$ is a homomorphism of partial groups and $\One'$ denotes the identity in $\L'$, then $\ker(\alpha):=\{f\in\L\colon f\alpha=\One'\}$ 
is called the kernel of $\alpha$.  
\end{itemize}
\end{definition}

If $\L'$ is a partial group, then $\L'\subseteq\D'$. Hence, a projection of partial groups $\L\rightarrow \L'$ is always surjective.

\smallskip

If $\alpha\colon \L\rightarrow \L'$ is a homomorphism of partial groups, then $\ker(\alpha)\unlhd\L$ by \cite[Lemma~1.14]{loc1}. If $(\L,\Delta,S)$ is a locality and $\N\unlhd\L$, then conversely, one can construct a partial group $\L'$ and a projection of partial groups $\L\rightarrow \L'$ with kernel $\N$. Namely, define a \emph{coset} of $\N$ in $\L$ to be a subset of the form 
\[\N f:=\{\Pi(n,f)\colon n\in\N,\;f\in\L\mbox{ such that }(n,f)\in\D\}.\]
Call a coset \emph{maximal} if it is maximal with respect to inclusion among all cosets of $\N$ in $\L$. Write $\L/\N$ for the set of maximal cosets of $\N$ in $\L$. By \cite[Proposition~3.14(d)]{loc1}, $\L/\N$ forms a partition of $\L$. Thus, there is a natural map
\[\alpha\colon\L\rightarrow \L/\N\]
which sends every element $f\in\L$ to the unique maximal coset containing $f$. It turns out that the set $\L/\N$ can be (in a unique way) given the structure of a partial group such that the map $\alpha$ above is a projection of partial groups (cf. \cite[Lemma~3.16]{loc1}). We call the map $\alpha$ therefore the \emph{natural projection} from $\L$ to $\L/\N$. The identity element of this partial group is the maximal coset $\N=\N\One$. In particular, $\N$ is the kernel of the natural projection.

\smallskip

Given a locality $(\L,\Delta,S)$ and $\N\unlhd\L$, we will sometimes write $\ov{\L}$ for the set $\L/\N$. We mean then implicitly that we use a kind of ``bar notation'' similarly as it is commonly used for groups. Namely, if $X$ is an element or subset of $\L$, then $\ov{X}$ denotes the image of $X$ under the natural projection $\alpha\colon\L\rightarrow \ov{\L}$. Later on, we will consider cases where $\N\cap S=\One$. Since $\N$ equals the kernel of $\alpha$, the restriction $\alpha|_S\colon S\rightarrow \ov{S}$ is then a bijection. Thus, we will be able to identify $S$ with $\ov{S}$.

\begin{definition}
Let $(\L, \Delta, S)$ and $(\L', \Delta', S')$ be localities and let $\alpha \colon \L \rightarrow \L'$ be a projection of partial groups.  We say that $\alpha$ is a \emph{projection of localities} from $(\L,\Delta,S)$ to $(\L',\Delta',S')$ if the set $\Delta\alpha:= \{P\alpha \mid P \in \Delta \}$ equals $\Delta'$.
\end{definition}

If $\alpha$ is a projection of localities from $(\L,\Delta,S)$ to $(\L',\Delta',S')$, then notice that $\alpha$ maps $S$ to $S'$, as $S$ and $S'$ are the unique maximal elements of $\Delta$ and $\Delta'$ respectively.

\subsection{Localities with special properties}

The following definition was partly introduced already in the introduction. Recall that a finite group $G$ is of characteristic $p$ if $C_G(O_p(G))\leq O_p(G)$. 

\begin{definition} Let $(\L,\Delta,S)$ be a locality. 
\begin{itemize}
 \item We call $(\L,\Delta,S)$ \emph{cr-complete} if $\F_S(\L)^{cr}\subseteq \Delta$.
 \item A locality $(\L,\Delta,S)$ is said to be of \emph{objective characteristic $p$} if $N_\L(P)$ is of characteristic $p$ for every $P \in \Delta$.
\item A locality is called a \emph{linking locality} if it is cr-complete and of objective characteristic $p$.
\end{itemize}
\end{definition}

\begin{prop}\label{P:crCompleteSaturated}
Let $(\L,\Delta,S)$ be a locality. 
\begin{itemize}
 \item [(a)] $\F_S(\L)$ is $\Delta$-generated and $\Delta$-saturated.
 \item [(b)] If every $\F_S(\L)$-critical subgroup is an element of $\Delta$, then $\F_S(\L)$ is saturated and $(\L,\Delta,S)$ is cr-complete.
 \item [(c)] If $(\L,\Delta,S)$ is cr-complete, then $\F_S(\L)$ is saturated. Hence, the fusion system of a linking locality is saturated.
\end{itemize}
\end{prop}

\begin{proof}
The fusion system $\F_S(\L)$ is $\Delta$-generated by definition and $\Delta$-saturated by \cite[Lemma~2.9]{loc1} and \cite[Lemma~3.27]{Henke:Regular}. Thus, Theorem~\ref{T:BCGLOControlling} implies that $\F_S(\L)$ is saturated if $\Delta$ contains every $\F_S(\L)$-critical subgroup. Part (b) follows now from Lemma~\ref{L:radical}(b) and part (c) from  Lemma~\ref{L:radical}(a). Alternatively, (c) is proved in \cite[Theorem~3.26]{Henke:Regular}.
\end{proof}

It is of crucial importance that cr-complete localities (and thus linking localities) lead to saturated fusion systems. This fact was first observed by Chermak. Conversely, it follows from the existence and uniqueness of centric linking systems (see \cite{Ch,Oliver:2013,Glauberman/Lynd}) that there is a linking locality attached to every saturated fusion system. In this context, the following definition plays a crucial role.

\begin{definition}[{\cite[Definition~1, Lemma~3.1]{subcentric}}]\label{D:Subcentric}
Let $\F$ be a saturated fusion system. A subgroup $P\leq S$ is called \emph{$\F$-subcentric} if $N_\F(Q)$ is constrained for every fully $\F$-normalized $\F$-conjugate $Q$ of $P$.  We denote by  $\F^s$ the set of $\F$-subcentric subgroups of $S$.
\end{definition}

It can be shown that, for every $\F$-closed set $\Delta$ of subgroups of $S$ with $\F^{cr}\subseteq\Delta\subseteq \F^s$, there exists an (essentially unique) linking locality over $\F$ with object set $\Delta$. Moreover, $\F^s$ is $\F$-closed and thus there is an (essentially unique) linking locality over $\F$ with object set $\F^s$, which can be seen as the ``largest'' linking locality over $\F$. The reader is referred to  \cite[Theorem~A]{subcentric} for a precise statement of these results.

\subsection{Regular localities}\label{SS:Regular}

Regular localities were introduced by Chermak \cite{loc3}, but we will refer to the treatment of the subject in \cite{Henke:Regular}. If $(\L,\Delta,S)$ is a linking locality over a saturated fusion system $\F$, then there is the \emph{generalized Fitting subgroup} $F^*(\L)$ of $\L$ defined as a certain partial normal subgroup of $\L$ (cf. \cite[Definition~3]{Henke:Regular}). It turns out that the set 
$\{P\leq S\colon P\cap F^*(\L)\in\F^s\}$ depends only on $\F$ and not on the linking locality $(\L,\Delta,S)$ (see \cite[Lemma~10.2]{Henke:Regular}). Thus, the following definition makes sense.
\begin{definition}
If $\F$ is a saturated fusion system and $(\L,\Delta,S)$ is a linking locality over $\F$, then set
\[\delta(\F):=\{P\leq S\colon P\cap F^*(\L)\in\F^s\}.\]
A linking locality $(\L,\Delta,S)$ is called a \emph{regular locality} if $\Delta=\delta(\F)$.  
\end{definition}

For every saturated fusion system $\F$, the set $\delta(\F)$ is $\F$-closed and thus there exists a regular locality over $\F$ (cf. \cite[Lemma~10.4]{Henke:Regular}).

\smallskip

It turns out that regular localities have particularly nice properties. To describe these, suppose that $(\L,\Delta,S)$ is a regular locality. If $\H$ is a partial normal subgroup of $\L$ or, more generally, a partial subnormal subgroup of $\L$, then by \cite[Theorem~3, Corollary~10.19]{Henke:Regular}, $\E:=\F_{S\cap \H}(\H)$ is saturated and $(\H,\delta(\E),S\cap\H)$ is a regular locality over $\E$. This leads to a natural notion of components of $\L$ (see \cite[Definition~11.1]{Henke:Regular}) and to a theory of components of regular localities which mirrors results from finite group theory (see \cite[Chapter~11]{Henke:Regular}). In particular, the product $E(\L)$ of components of $\L$ forms a partial normal subgroup with $F^*(\L)=E(\L)O_p(\L)$. 

\smallskip

If $(\L,\Delta,S)$ is a regular locality over $\F$ (or more generally a linking locality over $\F$ with $\delta(\F)\subseteq\Delta$), then it is shown in \cite[Theorem~A]{normal} that the assignment $\N\mapsto \F_{S\cap\N}(\N)$ defines a bijection $\Phi$ from the set of partial normal subgroups of $\L$ to the set of normal subsystems of $\F$. This bijection sends $F^*(\L)$ to $F^*(\F)$ and $E(\L)$ to $E(\F)$. Moreover, in the case that $(\L,\Delta,S)$ is regular, we have $C_\L(\N)\unlhd\L$ and $\Phi(C_\L(\N))=C_\F(\Psi(\N))$ (see \cite[Theorem~3]{Henke:Regular} and \cite[Proposition~6.7]{normal}). We refer the reader here to \cite[p.3]{AschbacherGeneralized} or \cite[Definition~7.2]{normal} for the definitions of $F^*(\F)$ and $E(\F)$, and to \cite[Chapter~6]{AschbacherGeneralized} or \cite{centralizer.normal} for a definition of the centralizer $C_\F(\E)$ of a normal subsystem $\E$. 

\smallskip

The results we just summarized are used to prove our next lemma, which collects most of the information we need in Section~\ref{S:Products}. It will be convenient to use the following notation.

\begin{notation}\label{N:PcapE}
If $\F$ is a fusion system over $S$ and $\E$ is a subsystem of $\F$ over $T\leq S$, then we set $\E\cap P:=T\cap P$ for every subgroup $P\leq S$.
\end{notation}

\begin{lemma}\label{L:T0R}
Let $(\L,\Delta,S)$ be a regular locality over $\F$ and $\N\unlhd\L$. Set $T:=\N\cap S$, $\E:=\F_T(\N)$, $T^*:=F^*(\L)\cap S$, $T_0:=E(\N)\cap S$ and $R:=E(C_\L(\N))\cap S$. Then the following hold:
\begin{itemize}
 \item [(a)] $N_\L(T^*)=N_\L(T_0)\cap N_\L(R)$ and $N_\N(T^*)=N_\N(T_0)$.
 \item [(b)] We have $T_0=E(\E)\cap S$ and $R=E(C_\F(\E))\cap S$. 
 \item [(c)] $N_\E(T_0)=\F_T(N_\N(T_0))$. 
 \item [(d)] $\delta(\E)$ is closed under passing to $\L$-conjugates in $S$.
\end{itemize}
\end{lemma}

\begin{proof}
By \cite[Theorem~3]{Henke:Regular}, we have $C_\L(\N)\unlhd\L$; moreover $(\N,\delta(\E),T)$ is a regular locality over $\E$ and $C_\L(\N)$ can also be given the structure of a regular locality. In particular, $E(\N)$ and $E(C_\L(\N))$ are well-defined.

\smallskip

\textbf{(a)} By \cite[Lemma~11.13]{Henke:Regular}, $E(\N)\unlhd\L$ and similarly $E(C_\L(\N))\unlhd\L$. Moreover, by \cite[Lemma~11.16]{Henke:Regular}, $E(\L)=E(\N)E(C_\L(\N))$, which  by \cite[Theorem~1]{Henke:2015a} implies $T_0R=E(\L)\cap S\leq T^*$. In particular, $N_\L(T^*)$ normalizes $T_0=T^*\cap E(\N)$ and $R=T^*\cap E(C_\L(\N))$. On the other hand, by \cite[Lemma~11.9]{Henke:Regular}, we have $T^*=(E(\L)\cap S)O_p(\L)$ and so $N_\L(T_0)\cap N_\L(R)\leq N_\L(T^*)$. Thus, $N_\L(T^*)=N_\L(T_0)\cap N_\L(R)$. As $R\subseteq C_\L(\N)$, it follows from \cite[Lemma~3.5]{Henke:Regular} that $\N\subseteq C_\L(R)\subseteq N_\L(R)$. Hence, we can conclude that $N_\N(T_0)=\N\cap N_\L(T_0)\cap N_\L(R)=\N\cap N_\L(T^*)=N_\N(T^*)$. This proves (a). 

\smallskip

\textbf{(b)} By \cite[Proposition~6.7]{normal}, we have $\F_{C_S(\N)}(C_\L(\N))=C_\F(\E)$. This means that $(C_\L(\N),\delta(C_\F(\E)),C_S(\N))$ is a regular locality over $C_\F(\E)$. Recall also that $(\N,\delta(\E),T)$ is a regular locality over $\E$. Hence, \cite[Theorem~E(d)]{normal} yields (b).

\smallskip

\textbf{(c)} By \cite[Corollary~11.10]{Henke:Regular}, $T_0\in\delta(\E)$. Hence, (c) follows from Lemma~\ref{L:NDPNHP}(b) applied with $(\N,\delta(\E),T)$ in place of $(\L,\Delta,S)$.

\smallskip

\textbf{(d)} The set $\delta(\E)$ is $\E$-closed by \cite[Lemma~10.4]{Henke:Regular}. Moreover, it follows from \cite[Theorem~10.16(f)]{Henke:Regular} that  $\delta(\E)$ is closed under passing to $N_\L(T)$-conjugates. By Lemma~\ref{L:FrattiniSplitting}, for every $g\in\L$ there exist $n\in\N$ and $f\in N_\L(T)$ such that $(n,f)\in\D$, $g=nf$ and $S_g=S_{(n,f)}$. Now (d) follows using Lemma~\ref{L:LocalitiesProp}(c),(e).
\end{proof}

\section{Some additional lemmas}

\subsection{Lemmas on groups}

\begin{lemma}\label{lem:sylow.prod}
Let $G$ be a finite group, $S\in \Syl_p(G)$ a Sylow $p$-subgroup of $G$ and $N\norm G$ a normal subgroup of $G$. If $H\leq G$ is such that $S \cap H \in \Syl_p(H)$, then \[(S \cap N)(S \cap H) \in \Syl_p(NH).\]
\end{lemma}

\begin{proof}
Notice that $P:=(S \cap N)(S \cap H)\leq S \cap NH$ and $|P| = \frac{|S \cap N||S \cap H|}{|S \cap N \cap H|}$. By assumption $S \cap H \in \Syl_p(H)$. As $N\unlhd G$ and $N\cap H\unlhd H$, it follows $S\cap N\in\Syl_p(N)$ and $S\cap N\cap H\in\Syl_p(N\cap H)$. The assertion follows now from $|NH|=\frac{|N|\;|H|}{|N\cap H|}$. 
\end{proof}

Recall that a finite group $G$ is said to be of characteristic $p$ if $ C_G(O_p(G)) \leq O_p(G)$.

\begin{lemma}\label{prop.char.p.group}
Let $G$ be a finite group of characteristic $p$. Then the following hold:
\begin{itemize}
\item[(a)] $N_G(P)$ and $ C_G(P)$ are of characteristic $p$ for all non-trivial $p$-subgroups $P \leq G$.
\item[(b)] Every subnormal subgroup of $G$ is of characteristic $p$.
\item[(c)] If $N$ is a normal subgroup of $G$ of $p$-power index, then $G$ is of characteristic $p$ if and only if $N$ is of characteristic $p$.
\end{itemize}
\end{lemma}

\begin{proof}
Let $P$ be a $p$-subgroup of $G$. By \cite[Lemma 1.2(a),(c)]{MS}, every subnormal subgroup of $G$ is of characteristic $p$ and  $N_G(P)$ is of characteristic $p$. In particular, $ C_G(P) \norm N_G(P)$ is of characteristic $p$. 

\smallskip

Let now $N$ be as in (c). By \cite[Lemma~1.3]{MS}, $G$ is of characteristic $p$ if and only if $O^p(G)$ is of characteristic $p$. Similarly, $N$ is of characteristic $p$ if and only if $O^p(N)$ is of characteristic $p$. As $O^p(N)=O^p(G)$, this implies (c).
\end{proof}

\begin{lemma}\label{L:CharpCharacterize}
Let $G$ be a finite group, and let $N$ be a normal subgroup of $G$ with Sylow $p$-subgroup $T$. Then the following are equivalent:
\begin{itemize}
 \item [(i)] $G$ is of characteristic $p$.
 \item [(ii)] $N$ and $N_G(T)$ are of characteristic $p$. 
 \item [(iii)] $N$ and $C_G(N)$ are of characteristic $p$.
\end{itemize}
\end{lemma}

\begin{proof}
It follows from Lemma~\ref{prop.char.p.group} that (i) implies (ii). As $N$ is normal in $G$, $C_G(N)$ is also normal in $G$. Notice that $C_G(N)\leq C_G(T)\leq N_G(T)$. Thus $C_G(N)$ is a normal subgroup of $N_G(T)$. So by Lemma~\ref{prop.char.p.group}(b), $C_G(N)$ has characteristic $p$ if $N_G(T)$ has characteristic $p$. Hence, (ii) implies (iii). Assume now that (iii) holds. It follows from \cite[6.5.2]{KS} that every component of $G$ is either a component of $N$ or of $C_G(N)$. As $N$ and $C_G(N)$ do not have any components, it follows that $E(G)=1$. Moreover, $[N,O_{p^\prime}(G)]\leq N\cap O_{p^\prime}(G)=O_{p^\prime}(N)=1$. Thus $O_{p^\prime}(G)\leq C_G(N)$ and so $O_{p^\prime}(G)=O_{p^\prime}(C_G(N))=1$. Hence, (i) holds and the proof is complete. 
\end{proof}

\subsection{Some properties of fusion systems}

It will be convenient to use the following definition. 

\begin{definition}
Suppose $\F$ and $\tF$ are fusion systems over $S$ and $\tS$ respectively.
\begin{itemize}
\item If $\alpha\colon S\rightarrow \tS$ is a group isomorphism, then $\alpha$ is said to \emph{induce an isomorphism from $\F$ to $\tF$}, if $\{\alpha^{-1} \phi\alpha\colon \phi\in\Mor_\F(P,Q)\}=\Mor_{\tF}(P\alpha,Q\alpha)$ for all $P,Q\leq S$. 
\item We say that $\alpha\in\Aut(S)$ induces an automorphism of $\F$ if $\alpha$ induces an isomorphism from $\F$ to $\F$. Write $\Aut(\F)$ for the set of all $\alpha\in\Aut(S)$ which induce an automorphism of $\F$.  
\end{itemize}
\end{definition}

\begin{lemma}\label{L:crsqAutClosed}
Let $\F$ be a fusion system over $S$ and let $\E$ be an $\F$-invariant subsystem of $\F$. Then $\E^c$, $\E^r$ and $\E^{cr}$ are  closed under $\F$-conjugacy. Similarly, the set of $\E$-critical subgroups is closed under $\F$-conjugacy.
\end{lemma}

\begin{proof}
Let $T\leq S$ be such that $\E$ is a subsystem over $T$. Let $\Gamma$ be one of the sets $\E^c$, $\E^r$, $\E^{cr}$ or the set of $\E$-critical subgroups. It follows directly from the definition of these sets that $\Gamma$ is closed under $\E$-conjugacy. By the Frattini condition for $\F$-invariant subsystems \cite[Definition~I.6.1]{AKO}, it is thus sufficient to argue that $\Gamma$ is $\Aut_\F(T)$-invariant. The definition of $\F$-invariant subsystems implies moreover that every element of $\Aut_\F(T)$ induces an automorphism of $\E$. Therefore, $\alpha\in\Aut_\F(T)$ maps $\E$-conjugacy classes to $\E$-conjugacy classes and fully $\E$-normalized subgroups to fully $\E$-normalized subgroups. Moreover, if $P\leq T$ with $C_T(P)\leq P$, then $C_T(P\alpha)\leq P\alpha$. So $\E^c$ is $\Aut_\F(T)$-invariant. Notice also that $\alpha|_{N_T(Q)}$ induces an isomorphism from $N_\E(Q)$ to $N_\E(Q\alpha)$ for every $Q\leq T$. Hence, $\alpha$ maps $\E$-radical subgroups to $\E$-radical subgroups. So $\E^r$ and $\E^{cr}=\E^c\cap\E^r$ are $\Aut_\F(T)$-invariant. For $Q\leq T$, the map $\Aut_\E(Q)\rightarrow \Aut_\E(Q\alpha),\phi\mapsto \alpha^{-1}\phi\alpha$ is an isomorphism which takes $\Inn(Q)$ to $\Inn(Q\alpha)$ and $\Aut_T(Q)$ to $\Aut_T(Q\alpha)$. Hence, the set of $\E$-critical subgroups is $\Aut_\F(T)$ invariant. This proves the assertion. 
\end{proof}

\textbf{For the remainder of this subsection let $\F$ be a saturated fusion system over a $p$-group $S$.}

\smallskip

In the proofs of the following two lemmas we cite \cite{AOV1}. It should be pointed out that normal subsystems in the sense of \cite[Definition~1.18]{AOV1} correspond to weakly normal subsystems in the language of \cite{AKO} (i.e. in the language that we are using in this paper).

\begin{lemma}\label{L:RadicalIntersect}
Let $\E$ be a weakly normal subsystem of $\F$ over $T$. If $R\in\F^{cr}$, then $R\cap T\in\E^{cr}$.
\end{lemma}

\begin{proof}
This is \cite[Lemma~1.20(d)]{AOV1}, but follows also easily from the results stated before. Namely, by Corollary~\ref{C:FullyNormalizedConjugate}, there exists $\alpha\in\Hom_\F(N_S(R\cap T),S)$ such that $(R\cap T)^\alpha$ is fully normalized. Since $R\leq N_S(R\cap T)$, the subgroup $R^\alpha$ is well-defined. As $T$ is strongly closed, $R^\alpha\cap T=(R\cap T)^\alpha$. By Lemma~\ref{L:crsqAutClosed}, $\E^{cr}$ and $\F^{cr}$ are closed under $\F$-conjugacy. So replacing $R$ by $R^\alpha$, we may assume that $R\cap T$ is fully $\F$-normalized. If $Q$ is an $\E$-conjugate of $R\cap T$, then by Corollary~\ref{C:FullyNormalizedConjugate}, there is an $\F$-morphism mapping $N_S(Q)$ into $N_S(R\cap T)$ and thus $N_T(Q)$ into $N_T(R\cap T)$. Hence, $R\cap T$ is fully $\E$-normalized.  Now the claim follows from Lemma~\ref{L:radical}(b) and Lemma~\ref{L:FinvariantCritical}.
\end{proof}

\begin{lemma}\label{L:AOV}
Let $\E$ be an $\F$-invariant (not necessarily saturated) subsystem of $\F$ over $T\leq S$. Let $P$ be an $\E$-critical subgroup of $T$ such that  $P$ is fully $\F$-normalized. Then there exists $R\in \F^{cr}$ with $R\cap T=P$. 
\end{lemma}

\begin{proof}
This is a special case of \cite[Lemma~1.19]{AOV1}.
\end{proof}

We will now look at some properties of the set $\delta(\F)$ introduced in Subsection~\ref{SS:Regular}. As shown in \cite[Lemma~7.21]{normal}, $\delta(\F)$ can be characterized as the set of subgroups of $S$ containing an element of $F^*(\F)^s$.  The reader is referred to \cite{AschbacherGeneralized} or \cite[Definition~7.2]{normal} for the definition of the \emph{generalized Fitting subsystem} $F^*(\F)$ and the \emph{layer} $E(\F)$ of $\F$. Recall Notation~\ref{N:PcapE}.

\begin{lemma}\label{L:Regular}
Let $P\leq S$ with $O_p(\F)\leq P$. Then $P\in \delta(\F)$ if and only if $E(\F)\cap P\in E(\F)^s$.
\end{lemma}

\begin{proof}
We use that $F^*(\F)=E(\F)*\F_{O_p(\F)}(O_p(\F))$ by \cite[Theorem~7.10(e)]{normal}. In particular, $F^*(\F)\cap S=(E(\F)\cap S)O_p(\F)$. As $O_p(\F)\leq P$, a Dedekind Argument yields thus $F^*(\F)\cap P=( E(\F)\cap P)O_p(\F)$.  It is shown in \cite[Lemma~7.21]{normal} that $P\in\delta(\F)$ if and only if $F^*(\F)\cap P\in F^*(\F)^s$. The assertion follows now from \cite[Lemma~2.14(g)]{Henke:Regular}.
\end{proof}



\begin{lemma}\label{L:deltaFR}
Let $\E$ be a normal subsystem of $\F$ and $O_p(\F)\leq P\leq S$. Then the following hold:
\begin{itemize}
 \item [(a)] If $E(\F)=E(\E)$, then $P\in\delta(\F)$ if and only if $ \E\cap P\in\delta(\E)$.
 \item [(b)] Suppose $E(C_\F(\E))\cap S\leq P$. Then
\[P\in\delta(\F)\Longleftrightarrow E(\E)\cap P\in E(\E)^s\Longleftrightarrow \E\cap P\in\delta(\E).\]
\end{itemize}
\end{lemma}

\begin{proof}
By Lemma~\ref{L:Regular}, we have $P\in\delta(\F)$ if and only if $E(\F)\cap P\in E(\F)^s$. As $O_p(\E)\leq O_p(\F)\leq P$ by \cite[Lemma~2.12(b)]{subcentric}, it follows similarly that $\E\cap P\in \delta(\E)$ if and only if $E(\E)\cap P\in E(\E)^s$. This implies (a). By \cite[Lemma~7.13(c)]{normal}, $E(\F)=E(\E)*E(C_\F(\E))$. As $R:=E(C_\F(\E))\cap S\leq P$, we have $E(\F)\cap P=(E(\E)\cap P)R$. Notice that $R\in E(C_\F(\E))^s$. Hence, by \cite[Lemma~2.14(g)]{Henke:Regular}, $E(\F)\cap P\in E(\F)^s$ if and only if $E(\E)\cap P\in E(\E)^s$. This implies (b).
\end{proof}

\begin{remark}\label{R:ER}
If $\E$ is a normal subsystem of $\F$ defined over $T\leq S$ and $R$ is a subgroup of $S$, then there is a unique saturated subsystem $\mD$ of $\F$ over $TR$ with $O^p(\mD) = O^p(\E)$. It is denoted by $(\E R)_\F$ (sometimes denoted simply $\E R$). This was first shown by Aschbacher \cite[Chapter 8]{AschbacherGeneralized}. The result was revisited in \cite{Henke.products} where a concrete description is given.
\end{remark}

We conclude this section with the following lemma, which was first proved by Aschbacher \cite[1.3.2]{AschbacherFSCT}. We give a new proof using localities.

\begin{lemma}\label{L:FusionFrattini}
Let $\E$ be a normal subsystem of $\F$ over $T\leq S$. Then $\F=\< \E S, N_\F(T)\>$.
\end{lemma}

\begin{proof}
By \cite[Theorem~A]{subcentric} there exists a linking locality $(\L,\Delta,S)$ over $\F$ with $\Delta=\F^s$. Moreover, by \cite[Theorem~A]{normal}, there exists a partial normal subgroup $\N$ of $\L$ with $T=S\cap \N$ and $\E=\F_T(\N)$. Now, by \cite[Corollary~4.10]{normal}, we have $\E S=\F_S(\N S)$. As $\F=\F_S(\L)$ is generated by maps of the form $c_g\colon S_g\rightarrow S$ with $g\in\L$, it is sufficient to prove that such maps are in $\<\E S,N_\F(T)\>$. Let $g\in\L$. By Lemma~\ref{L:FrattiniSplitting}, there exist $n\in\N$ and $f\in N_\L(T)$ such that $(n,f)\in\D$, $g=nf$ and $P:=S_g=S_{(n,f)}$. Hence, $c_g=(c_n|_P)(c_f|_{P^n})$. Notice that $c_n|_P\in \F_S(\N S)=\E S$ and that $c_f|_{P^n}$ is a morphism in $N_\F(T)$. This shows the assertion. 
\end{proof}

\section{Kernels of localities}\label{S:Kernels}

In this section we investigate properties of kernels of localities. In particular, we prove Theorem~\ref{T:crComplete}, \ref{T:KernelLinkingLocality} and \ref{T:Regular}. Recall the following definition.

\begin{definition}
Let $(\L,\Delta,S)$ be a locality. A \emph{kernel} of $\L$ is a partial normal subgroup $\N$ of $\L$ such that $P\cap \N\in\Delta$ for every $P\in\Delta$. 
\end{definition}

\begin{lemma}\label{L:KernelBasic}
Let $\N$ be a kernel of a locality $(\L,\Delta,S)$. Set $\F=\F_S(\L)$, $T:=S\cap\N$, $\E=\F_T(\N)$ and $\Gamma:=\{P\in\Delta\colon P\leq T\}=\{P\cap T\colon P\in\Delta\}$. Then the following hold:
\begin{itemize}
 \item [(a)] $(\N,\Gamma,T)$ is a locality over $\E$ and $\Gamma$ is closed under $\F$-conjugacy.
 \item [(b)] For every $f\in N_\L(T)$, $\L=\D(f)$, $c_f\in\Aut(\L)$ and $c_f|_\N\in\Aut(\N)$.
 \item [(c)] $O_p(\N)=O_p(\L)\cap \N\unlhd\L$.
 \item [(d)] The subsystem $\E$ is $\F$-invariant. 
\end{itemize}
\end{lemma}

\begin{proof}
\textbf{(a,b)} It follows from \cite[Lemma~3.29(b),(c)]{Henke:Regular} and the definition of $\E$ that $(\N,\Gamma,T)$ is a locality over $\E$ and that (b) holds. Since $\Delta$ is closed under passing to $\L$-conjugates in $S$, $\N\unlhd\L$ and $\Gamma=\{P\in\Delta\colon P\subseteq\N\}$, it follows that $\Gamma$ is closed under passing to $\L$-conjugates in $S$. So $\Gamma$ is closed under $\F$-conjugacy and (a) holds.

\smallskip

\textbf{(c)} Notice that $O_p(\N)$ is invariant under conjugation by elements of $N_\L(T)$ as the elements of $N_\L(T)$ induce automorphisms of $\N$. Since $O_p(\N)\unlhd\N$, it follows from \cite[Corollary~3.13]{loc1} that $O_p(\N)\unlhd\L$ and thus $O_p(\N)\leq O_p(\L)$. Clearly $O_p(\L)\cap \N\leq O_p(\N)$ and so (c) holds.

\smallskip

\textbf{(d)} As $\Aut_\F(T)$ is generated by maps of the form $c_f|_T$ with $f\in N_\L(T)$ and these maps induce automorphisms of $\E$  by \cite[Lemma~3.29(d)]{Henke:Regular} , it follows that the elements of $\Aut_\F(T)$ induce automorphisms of $\E$. Now (d) is a consequence of  \cite[Lemma~3.28]{Henke:Regular}.
\end{proof}

\begin{definition}
We say that $(\N,\Gamma,T)$ is a kernel of $(\L,\Delta,S)$ to indicate that $\N$ is a kernel of $(\L,\Delta,S)$, $T=S\cap\N$ and $\Gamma=\{P\in\Delta\colon P\leq T\}=\{P\cap T\colon P\in\Delta\}$. 
\end{definition}

The following lemma can be seen as a converse to Lemma~\ref{L:KernelBasic}(a).

\begin{lemma}\label{L:ConstructLocality}
Suppose we are given a finite partial group $\L$ with product $\Pi\colon\D\rightarrow\L$. Let $S$ be a maximal $p$-subgroup of $\L$ and  $\N\unlhd\L$. Set $T:=\N\cap S$, let $\Gamma$ be a set of subgroups of $T$ and set $\Delta:=\{P\leq S\colon P\cap T\in\Gamma\}$. Suppose the following hold:
\begin{itemize}
 \item $\Gamma$ is closed under passing to overgroups and $\L$-conjugates which lie in $T$.
 \item $\D=\D_\Gamma$.
\end{itemize}
Then $(\L,\Delta,S)$ is a locality with kernel $(\N,\Gamma,T)$. 
\end{lemma}
 
\begin{proof}
If $(\L,\Delta,S)$ is a locality, then clearly $(\N,\Gamma,T)$ is a kernel of $(\L,\Delta,S)$. By Lemma~\ref{L:DDeltaDGamma}(b), we have $\D_\Delta=\D_\Gamma=\D$. Recall that $S$ is a maximal $p$-subgroup of $\L$ by assumption. As $\Gamma$ is overgroup-closed in $T$, it follows that $\Delta$ is overgroup-closed in $S$. Hence, it remains to show that $\Delta$ is closed under passing to $\L$-conjugates in $S$. 

\smallskip

Let $P\in\Delta$ and $f\in\L$ with $P\subseteq\D(f)$ and $P^f\subseteq S$. Observe that $Q:=P\cap \N\in\Gamma$ is normal in $P$ and $Q^f\in\Gamma$ by assumption. Hence, if $x,y\in P$, then $u:=(f^{-1},x,f,f^{-1},y,f)\in\D=\D_\Gamma$ via $Q^f$. Thus, by the axioms of a partial group, $x^fy^f=\Pi(u)=\Pi(f^{-1},x,y,f)=(xy)^f$ and so $c_f\colon P\rightarrow S,x\mapsto x^f$ is a group homomorphism. Therefore, $P^f$ is a subgroup of $S$. By Lemma~\ref{L:DDeltaDGamma}(a), $P^f\cap T=P^f\cap \N=Q^f\in\Gamma$ and so $P^f\in\Delta$.
\end{proof}

We will use from now on without further reference that, by Lemma~\ref{L:KernelBasic}(a), $(\N,\Gamma,T)$ is a locality if $(\N,\Gamma,T)$ is a kernel of a locality $(\L,\Delta,S)$. In particular, it makes sense to say that a kernel $(\N,\Gamma,T)$ is cr-complete.

\begin{prop}\label{Prop:kernel.cr.complete}
Let $(\L,\Delta,S)$ be a locality with kernel $(\N,\Gamma,T)$. Then $(\L,\Delta,S)$ is cr-complete if and only if $(\N,\Gamma,T)$ is cr-complete.
\end{prop}

\begin{proof}
Set $\E:=\F_T(\N)$ and $\F:=\F_S(\L)$. By Lemma~\ref{L:KernelBasic}(a),(d), $(\N,\Gamma,T)$ is a locality and $\E$ is $\F$-invariant. 

\smallskip

Assume first that $(\L,\Delta,S)$ is cr-complete. Then in particular, $\F$ is saturated by Proposition~\ref{P:crCompleteSaturated}(c). We show:
\begin{equation}\label{E:ECriticalinGamma}
 \mbox{If $Q\leq T$ is $\E$-critical, then $Q\in\Gamma$.}
\end{equation}
For the proof, let $Q$ be an $\E$-critical subgroup of $T$. Then there exists a fully $\F$-normalized $\F$-conjugate $P$ of $Q$, which by Lemma~\ref{L:crsqAutClosed} is also $\E$-critical. Hence, Lemma~\ref{L:AOV} implies the existence of an element $R\in\F^{cr}$ with $R\cap T=P$. As $\F^{cr}\subseteq\Delta$ and $(\N,\Gamma,T)$ is a kernel of $(\L,\Delta,S)$, it follows that $P\in\Gamma$. Hence, Lemma~\ref{L:KernelBasic}(a) yields $Q\in\Gamma$ and thus \eqref{E:ECriticalinGamma} holds. Proposition~\ref{P:crCompleteSaturated}(b) implies now that $\E$ is saturated and $(\N,\Gamma,T)$ is cr-complete.

\smallskip

To prove the other implication, assume that $(\N,\Gamma,T)$ is cr-complete. Then by Proposition~\ref{P:crCompleteSaturated}(c), $\E$ is saturated. Observe also that $\E^{cr}\subseteq\Gamma\subseteq\Delta$. As $\F$ is $\Delta$-saturated and $\Delta$-generated by Proposition~\ref{P:crCompleteSaturated}(a), it follows now from Theorem~\ref{T:FInvariantOliverExtensions2} that $\F$ is saturated. In particular, $\E$ is weakly normal in $\F$. So Lemma~\ref{L:RadicalIntersect} gives $R\cap T\in\E^{cr}\subseteq\Gamma\subseteq\Delta$ for every $R\in\F^{cr}$. As $\Delta$ is overgroup-closed, this implies that $(\L,\Delta,S)$ is cr-complete.
\end{proof}


\begin{lemma}\label{L:KernelBasiccd}
Let $(\N,\Gamma,T)$ be a kernel of a locality $(\L,\Delta,S)$. Set $\F:=\F_S(\L)$ and $\E:=\F_T(\N)$. Suppose $C_\N(T)\leq T$. Then the following hold:
\begin{itemize}
\item [(a)] $N_\N(T)\subseteq N_\N(TC_S(T))$.  
\item [(b)] If $\E$ and $\F$ are saturated, then $\E$ is normal in $\F$.
\end{itemize}
\end{lemma}

\begin{proof}
\textbf{(a)} Notice that $T\in\Gamma\subseteq\Delta$ and thus $N_\L(T)$ is a group. The assumption $C_\N(T)\leq T$ yields that  $N_\N(T)$ is of characteristic $p$.

\smallskip

Let $g\in N_\N(T)$ and set $P:=S_g$. Then $T\leq P$ and so $P^g=P$ by \cite[Lemma~3.1(b)]{loc1}. Hence, $g\in K:=N_\N(P)$. Note that $N_\L(P)\leq N_\L(T)$ and so $K=N_{N_\N(T)}(P)\leq N_\L(T)$. Calculating inside the group $N_\L(T)$, we have 
\[[K,N_{C_S(T)}(P)]\leq \N\cap C_\L(T)=C_\N(T)\leq T,\]
where the last inclusion holds by assumption. Thus $[g,N_{C_S(T)}(P)]\leq T$ and so $N_{C_S(T)}(P)\leq S_g=P$. Therefore, we have $N_{PC_S(T)}(P)=PN_{C_S(T)}(P)=P$. As $PC_S(T)$ is a $p$-group, this implies $C_S(T)\leq P$ and thus $g$ acts on $T$ and $C_S(T)$. Hence (a) holds.

\smallskip

\textbf{(b)} Notice that part (a) implies $[TC_S(T),N_\N(T)]\leq (TC_S(T))\cap\N=T$. Since $\Aut_\F(T)$ is generated by maps of the form $c_f|_T$ with $f\in N_\L(T)$, the extension condition as stated in \cite[Definition~I.6.1]{AKO} follows. Now (b) follows from Lemma~\ref{L:KernelBasic}(b).
\end{proof}

\begin{lemma}\label{L:ProjectionKernel}
A projection from a locality $(\L^*,\Delta^*,S^*)$ to a locality $(\L,\Delta,S)$ sends every kernel of $(\L^*,\Delta^*,S^*)$ to a kernel of $(\L,\Delta,S)$.   
\end{lemma}

\begin{proof}
Let $\varphi \colon \L^* \rightarrow \L$ be a projection of localities and let $\N^*$ be a kernel of $\L^*$.
Then $\N:=(\N^*)\varphi$ is a partial normal subgroup of $\L$ by \cite[Lemma~2.5]{Henke:NK}. Let $P \in \Delta$. Then $P=(P^*)\varphi$ for some $P^* \in\Delta^*$ since $\Delta^*\varphi=\Delta$. Since $\N^*$ is a kernel of $\L^*$ we have $\N^* \cap P^*\in \Delta^*$. Hence
\[ \N\cap P = (\N^*)\varphi \cap P^*\varphi\geq (\N^* \cap P^*)\varphi \in \Delta^*\varphi = \Delta.\]
As $\Delta$ is overgroup-closed, it follows that $\N\cap P\in\Delta$ and thus $\N$ is a kernel of $\L$.
\end{proof}

For the next lemma, the reader might want to recall Definition~\ref{D:RestrictionLocality}. By $\E^q$ we denote the quasicentric subgroups of a saturated fusion system $\E$; for the definition see \cite[Definition~III.4.5]{AKO}.

\begin{lemma}\label{L:ProjectionPartialNormal}
Let $(\L,\Delta,S)$ be a locality over $\F$ with cr-complete kernel $(\N,\Gamma,T)$. Set $\E:=\F_T(\N)$. Let $\E^{cr}\subseteq\Gamma_0\subseteq\Gamma\cap\E^q$ such that $\Gamma_0$ is $\E$-closed and $\Aut_\F(T)$-invariant, and let $\Delta_0$ be the set of overgroups in $S$ of the elements of $\Gamma_0$. Set \[\L_0:=\L|_{\Delta_0},\;\N_0:=\N\cap\L_0\mbox{ and }\Theta=\bigcup_{P\in\Gamma_0}O_{p^\prime}(N_\N(P)).\] 
Then the following hold:
\begin{itemize}
 \item[(a)] $(\L_0,\Delta_0,S)$ is a cr-complete locality over $\F$ with cr-complete kernel $(\N_0,\Gamma_0,T)$. Moreover, $\N_0=\N|_{\Gamma_0}$ and  $\F_T(\N_0)=\E$.
 \item[(b)] $\Theta$ is a partial normal subgroup of $\L_0$ with $\Theta\cap S=1$. In particular, setting $\ov{\L}_0=\L_0/\Theta$ and using the ``bar notation'' as usual, $S$ and $T$ are naturally isomorphic to $\ov{S}$ and $\ov{T}$ via the restriction of the natural projection map $\L_0\rightarrow\ov{\L}_0$. 
 \item[(c)] Identify $\ov{S}$ and $\ov{T}$ with $S$ and $T$. Then the following conditions hold:
\begin{itemize}
\item $(\ov{\L}_0,\Delta_0,S)$ is a locality over $\F$.
\item $(\ov{\N}_0,\Gamma_0,T)$ is a kernel of $(\ov{\L}_0,\Delta_0,S)$, which is a linking locality over $\E$.
\end{itemize}
\end{itemize}
\end{lemma}

\begin{proof}
\textbf{(a)} Recall from Lemma~\ref{L:KernelBasic}(d) that $\E$ is $\F$-invariant. As $\Gamma_0$ is $\Aut_\F(T)$-invariant and closed under $\E$-conjugacy, it follows from the Frattini condition for $\F$-invariant subsystems (cf. \cite[Definition~I.6.1]{AKO}) that $\Gamma_0$ is closed under $\F$-conjugacy. Thus, $\Delta_0$ is $\F$-closed, $\L_0$ is well-defined and $(\L_0,\Delta_0,S)$ is a locality. It is easy to observe that $\N_0$ is a partial normal subgroup of $\L_0$. It follows from the definition of $\Delta_0$ that $(\N_0,\Gamma_0,T)$ is a kernel of $(\L_0,\Delta_0,S)$ and $\N_0=\N|_{\Gamma_0}$ (see also Lemma~\ref{L:DDeltaDGamma}). As $\E^{cr}\subseteq \Gamma_0$,  we can conclude from Lemma~\ref{L:RadicalIntersect} that $\F^{cr}\subseteq \Delta_0\subseteq\Delta$. In particular, $\F$ and $\E$ are saturated by Proposition~\ref{P:crCompleteSaturated}(c). Note furthermore that, for all $P\in\Delta$,  $N_\L(P)=N_{\L_0}(P)$ and thus, by Lemma~\ref{L:NDPNHP}(a), $\Aut_\F(P)=\Aut_{\F_S(\L_0)}(P)$. Similarly, $\Aut_\E(Q)=\Aut_{\F_T(\N_0)}(Q)$ for all $Q\in\Gamma$. Hence, it follows from Alperin's Fusion Theorem \cite[Theorem~I.3.6]{AKO} (combined with Lemma~\ref{L:radical}(b)) that $\F=\F_S(\L_0)$ and $\E=\F_T(\N_0)$. This proves (a).

\smallskip

\textbf{(b,c)} We argue first that 
\begin{equation}\label{E:AlmostCharp}
\mbox{$N_{\N}(P)/O_{p^\prime}(N_{\N}(P))$ is of characteristic $p$ for all $P\in \Gamma_0$.} 
\end{equation}
By Lemma~\ref{L:LocalitiesProp}(b) we may reduce to the case that $P$ is fully $\E$-normalized. So let $P\in\Gamma_0$ be fully $\E$-normalized. As $\Gamma_0\subseteq \E^q$ \cite[Proposition~1(c)]{subcentric} gives that $C_\N(P)/O_{p^\prime}(C_\N(P))$ is a $p$-group. In particular, $C_\N(P)/O_{p^\prime}(C_\N(P))$ is of characteristic $p$ and thus  $C_\N(P)=C_{N_\N(P)}(P)$ is ``almost of characteristic $p$'' in the sense of \cite[Definition~2.6]{subcentric}. It follows therefore from \cite[Lemma~2.9]{subcentric} (applied with $N_\N(P)$ in place of $G$) that $N_\N(P)$ is ``almost of characteristic $p$'', which means that $N_\N(P)/O_{p^\prime}(N_\N(P))$ is of characteristic $p$. This proves \eqref{E:AlmostCharp}.

\smallskip

As $N_{\N_0}(P)=N_\N(P)$ for all $P\in\Gamma_0$, it follows from \eqref{E:AlmostCharp} and from \cite[Proposition~6.4]{subcentric} applied with $(\N_0,\Gamma_0,T)$ in place of $(\L,\Delta,S)$ that $\Theta$ is a partial normal subgroup of $\N_0$, that $T\cap\Theta=1$, that the canonical projection $\N_0\rightarrow \N_0/\Theta$ is injective on $T$, and that $(\N_0/\Theta,\Gamma_0,T)$ is a linking locality over $\E$ if we identify $T$ with its image in $\N_0/\Theta$. The elements of $N_\L(T)=N_{\L_0}(T)$ induce $\F$-automorphisms of $T$ and act thus by assumption on $\Gamma_0$ via conjugation. Hence, it follows from Lemma~\ref{L:LocalitiesProp}(b) that $\Theta$ is invariant under conjugation by elements of $N_{\L_0}(T)$.  Now \cite[Corollary~3.13]{loc1} gives that $\Theta$ is a partial normal subgroups of $\L_0$. Moreover, $S\cap\Theta=S\cap\N\cap\Theta=T\cap\Theta=1$. This means that the kernel of the natural projection $\rho\colon \L_0\rightarrow \ov{\L_0}:=\L_0/\Theta$ intersects trivially with $S$ and restricts thus to isomorphisms $S\rightarrow \ov{S}$ and $T\rightarrow \ov{T}$. Hence (b) holds. 

\smallskip

We use now the notation introduced in (b) and (c). By \cite[Theorem~4.3]{loc1}, $(\ov{\L}_0,\Delta_0,S_0)$ is a locality. It is a special case of \cite[Lemma~2.21(b)]{Henke:2020} that $\F_S(\ov{\L}_0)=\F_S(\L_0)=\F$. Using \cite[Lemma~3.15]{loc1} one can observe that $\rho|_{\N_0}$ coincides with the canonical projection $\N_0\rightarrow \N_0/\Theta$. Thus, $(\ov{\N}_0,\Gamma_0,T)=(\N_0/\Theta,\Gamma_0,T)$ is a linking locality over $\E$. Now (c) follows from Lemma~\ref{L:ProjectionKernel}.
\end{proof}

\begin{prop}\label{P:KernelcrLinking}
 Let $\F$ be a fusion system over $S$ with a subsystem $\E$ over $T$. Then the following conditions are equivalent:
\begin{itemize}
 \item [(i)] There exists a locality $(\L,\Delta,S)$ over $\F$ with a cr-complete kernel $(\N,\Gamma,T)$ over $\E$.
 \item [(ii)] There exists a locality $(\L,\Delta,S)$ over $\F$ with a kernel $(\N,\Gamma,T)$, which is a linking locality over $\E$.
\end{itemize}
If either of these two conditions holds, then $\E$ is a normal subsystem of $\F$.
\end{prop}

\begin{proof}
 Clearly (ii) implies (i) as every linking locality is cr-complete. Assume now that (i) holds. Notice then that $\E^{cr}\subseteq  \Gamma_0:=\Gamma\cap\E^c\subseteq \Gamma\cap\E^q$. Moreover, $\Gamma_0$ is $\E$-closed and $\Aut_\F(T)$ invariant as the same holds for $\Gamma$ and $\E^c$ (cf. Lemma~\ref{L:crsqAutClosed}). Hence, (ii) follows from Lemma~\ref{L:ProjectionPartialNormal}(c). 

\smallskip

Assume now (ii) holds. If $(\N,\Gamma,T)$ is a kernel of a locality $(\L,\Delta,S)$ over $\F$ such that $(\N,\Gamma,T)$ is a linking locality over $\E$, then $(\N,\Gamma,T)$ is in particular cr-complete. Thus, $(\L,\Delta,S)$ is cr-complete by Proposition~\ref{Prop:kernel.cr.complete}. Hence $\E$ and $\F$ are saturated by Proposition~\ref{P:crCompleteSaturated}(c). So $\E\unlhd\F$ by Lemma~\ref{L:KernelBasiccd}(b). This proves the assertion.
\end{proof}

\begin{proof}[Proof of Theorem~\ref{T:crComplete}]
The statement follows from Proposition~\ref{Prop:kernel.cr.complete} and Proposition~\ref{P:KernelcrLinking}.
\end{proof}

\begin{prop}\label{P:ObjectiveCharpCharacterize}
 Let $(\L,\Delta,S)$ be a locality with a kernel $(\N,\Gamma,T)$. Then the following conditions are equivalent:
\begin{itemize}
\item [(i)] $(\L,\Delta,S)$ is of objective characteristic $p$.
\item [(ii)] $N_\L(T)$ is of characteristic $p$ and $(\N,\Gamma,T)$ is of objective characteristic $p$.
\item [(iii)] $ C_\L(T)$ is of characteristic $p$ and $(\N,\Gamma,T)$ is of objective characteristic $p$.
\item [(iv)] $N_\L(P)$ is of characteristic $p$ for every $P\in\Gamma$.
\end{itemize}
\end{prop}

\begin{proof}
By Lemma~\ref{L:KernelBasic}(a), $(\N,\Gamma,T)$ is a locality. Notice that $T=S\cap\N\in\Gamma\subseteq\Delta$ and thus $N_\L(T)$ is a finite group with $C_{N_\L(T)}(T)=C_\L(T)$. So properties (ii) and (iii) are equivalent by \cite[Lemma~2.9]{subcentric}. Hence, it is sufficient to show that properties (i),(ii) and (iv) are equivalent. 

\smallskip

If (i) holds then $N_\L(P)$ is of characteristic $p$ for every $P\in\Gamma$. In particular, $N_\L(T)$ is of characteristic $p$. Moreover, by Lemma~\ref{prop.char.p.group}(b), $N_\N(P)\unlhd N_\L(P)$ is of characteristic $p$ for every $P\in\Gamma$ showing that $(\N,\Gamma,T)$ is of objective characteristic $p$. So (i) implies (ii).

\smallskip

To show that (ii) implies (iv), assume now that (ii) holds. Suppose furthermore that (iv) is false, i.e. there exists $P\in\Gamma$ such that $G:=N_\L(P)$ is not of characteristic $p$. Choose such $P$ of maximal order. By Lemma~\ref{L:LocalitiesProp}(b) and \cite[Lemma~2.9]{loc1}, we may replace $P$ by a suitable $\L$-conjugate of $P$ and assume that $N_S(P)$ is a Sylow $p$-subgroup of $G$. As (ii) holds, we have $P\neq T$ and thus $P<N_T(P)$. Hence, by the maximality of $|P|$, the group $N_\L(N_T(P))$ has characteristic $p$. Hence, by Lemma~\ref{prop.char.p.group}(a), $N_G(N_T(P))=N_{N_\L(N_T(P))}(P)$ is of characteristic $p$. As $(\N,\Gamma,T)$ is of objective characteristic $p$, we know also that $N:=N_\N(P)\unlhd G$ is of characteristic $p$. Notice that $N_T(P)=N_S(P)\cap N$ is a Sylow $p$-subgroup of $N$, as $N_S(P)$ is a Sylow $p$-subgroup of $G$. So it follows from Lemma~\ref{L:CharpCharacterize} that $G$ is of characteristic $p$ contradicting our assumption. Hence, (ii) implies (iv).

\smallskip

Assume now that (iv) holds. If $P\in\Delta$ is arbitrary, then $P\cap\N\in\Gamma$ and so $N_\L(P\cap\N)$ is of characteristic $p$. Observe that $N_\L(P)\subseteq N_\L(P\cap\N)$ and thus $N_\L(P)=N_{N_\L(P\cap\N)}(P)$ is of characteristic $p$ by Lemma~\ref{prop.char.p.group}(a). Hence, (iv) implies (i).  
\end{proof}

\begin{proof}[Proof of Theorem~\ref{T:KernelLinkingLocality}]
 The statement follows from Proposition~\ref{Prop:kernel.cr.complete} and Proposition~\ref{P:ObjectiveCharpCharacterize}. 
\end{proof}

\begin{proof}[Proof of Theorem~\ref{T:Regular}]
Set $\F:=\F_S(\L)$ and $\E:=\F_T(\N)$. By \cite[Lemma~3.28(c)]{normal}, $(\L,\Delta,S)$ is a linking locality if and only if $(\L,\tDelta,S)$ is a linking locality and similarly, $(\N,\Gamma,T)$ is a linking locality if and only if $(\N,\tGamma,T)$ is a linking locality.\footnote{It should be noted here that the term ``proper locality'' used in \cite{normal} means exactly the same as the term ``linking locality''.} Thus, supposing from now on that $(\L,\Delta,S)$ and $(\N,\Gamma,T)$ are linking localities, it is by Theorem~\ref{T:KernelLinkingLocality} enough to show that $\tDelta=\delta(\F)$ if and only if $\tGamma=\delta(\E)$ and that $E(\L)=E(\N)$.

\smallskip

By Proposition~\ref{P:KernelcrLinking}, $\E\unlhd\F$. Our assumption yields moreover that $N_\L(T)$ is of characteristic $p$. Hence, by Lemma~\ref{L:NDPNHP}(b), $N_\F(T)=\F_S(N_\L(T))$ is constrained. It follows then from \cite[Corollary~7.18, Lemma~7.19]{normal} that $E(C_\F(\E))=E(N_\F(T))$ is the trivial fusion system. By \cite[Lemma~7.13(c)]{normal} this implies $E(\F)=E(\E)$. 

\smallskip

By \cite[Lemma~11.13]{Henke:Regular}, $E(\N)\unlhd\L$. Moreover, by \cite[Theorem~E(d)]{normal}, $\F_{S\cap E(\N)}(E(\N))=E(\E)=E(\F)$ and $E(\L)$ is the unique partial normal subgroup of $\L$ with $\F_{S\cap E(\L)}(E(\L))=E(\F)$. Hence, $E(\N)=E(\L)$.

\smallskip

By Lemma~\ref{L:OpL} and \cite[Proposition~5]{subcentric}, $O_p(\F)=O_p(\L)$ and $O_p(\E)=O_p(\N)$. Moreover, $O_p(\L)\cap T=O_p(\L)\cap\N=O_p(\N)$ by Lemma~\ref{L:KernelBasic}(c).  

\smallskip

Assume first that $\tGamma=\delta(\E)$. Let $P\leq S$ with $O_p(\L)=O_p(\F)\leq P$. Then $O_p(\N)\leq P\cap T$ and thus $P\cap T\in\Gamma$ if and only if $P\cap T\in\tGamma$. Hence,
\begin{eqnarray*}
P\in\delta(\F) &\Longleftrightarrow& P\cap T\in\delta(\E) \mbox{ (by Lemma~\ref{L:deltaFR}(a))}\\
&\Longleftrightarrow & P\cap T\in\tGamma \mbox{ (as $\tGamma=\delta(\E)$)}\\
&\Longleftrightarrow & P\cap T\in\Gamma \mbox{ (as noted above)}\\
&\Longleftrightarrow & P\in\Delta\mbox{ (since $(\N,\Gamma,T)$ is a kernel).}
\end{eqnarray*}
Since the above equivalences hold for every $P\leq S$ with $O_p(\L)\leq P$, it follows now from \cite[Lemma~10.6]{Henke:Regular} that $\delta(\F)=\{P\leq S\colon PO_p(\L)\in\delta(\F)\}=\{P\leq S\colon PO_p(\L)\in\Delta\}=\tDelta$. This proves one direction.

\smallskip

Assume now the other way around that $\delta(\F)=\tDelta$. Let $Q\leq T$ with $O_p(\N)\leq Q$. Then $QO_p(\L)\cap T=Q(O_p(\L)\cap T)=QO_p(\N)=Q$. Hence 
\begin{eqnarray*}
 Q\in\delta(\E) &\Longleftrightarrow& QO_p(\L)\cap T\in\delta(\E)\mbox{ (as $Q=QO_p(\L)\cap T$)}\\
&\Longleftrightarrow & QO_p(\L)\in\delta(\F)\mbox{ (by Lemma~\ref{L:deltaFR}(a) and since $O_p(\L)=O_p(\F)$)}\\
&\Longleftrightarrow & QO_p(\L)\in\tDelta\mbox{ (as $\delta(\F)=\tDelta$)}\\
&\Longleftrightarrow & QO_p(\L)\in\Delta\mbox{ (by definition of $\tDelta$)}\\
&\Longleftrightarrow & QO_p(\L)\cap T\in\Gamma\mbox{ (as $(\N,\Gamma,T)$ is a kernel)}\\
&\Longleftrightarrow & Q\in \Gamma \mbox{ (as $Q=QO_p(\L)\cap T$)}.
\end{eqnarray*}
The above equivalences hold for all $Q\leq T$ with $O_p(\N)\leq Q$. Hence, it follows from \cite[Lemma~10.6]{Henke:Regular} that $\delta(\E)=\{Q\leq T\colon QO_p(\N)\in\delta(\E)\}=\{Q\leq T\colon  QO_p(\N)\in\Gamma\}=\tGamma$. This proves the assertion.
\end{proof}

\begin{lemma}\label{L:KernelRegularNLT0}
 Let $(\L,\Delta,S)$ be a locality with kernel $(\N,\Gamma,T)$ such that $(\N,\Gamma,T)$ is a regular locality. Set $\E:=\F_T(\N)$, $T_0:=E(\N)\cap S=E(\E)\cap S$ and $\tDelta:=\{P\leq S\colon PO_p(\L)\in\Delta\}$. Then the following are equivalent:
\begin{itemize}
 \item [(i)] $(\L,\Delta,S)$ is a linking locality.
 \item [(ii)] $(\L,\tDelta,S)$ is a regular locality.
 \item [(iii)] $N_\L(T)$ is a group of characteristic $p$.
 \item [(iv)] $N_\L(T_0)$ is a group of characteristic $p$.
\end{itemize}
\end{lemma}

\begin{proof}
Set $\F:=\F_S(\L)$. By \cite[Theorem~E(d)]{normal}, $E(\E)=\F_{E(\N)\cap S}(E(\N))$ and in particular $T_0:=E(\N)\cap S=E(\E)\cap S$. As $(\N,\Gamma,T)$ is regular (and in particular a linking locality), it follows from Theorem~\ref{T:KernelLinkingLocality} that (i) and (iii) are equivalent. Moreover, by \cite[Lemma~10.6]{Henke:Regular}, $\Gamma=\delta(\E)=\{Q\leq T\colon QO_p(\N)\in\Gamma\}$. Therefore, properties (ii) and (iii) are equivalent by Theorem~\ref{T:Regular}. Hence, it remains to prove that (iii) and (iv) are equivalent.

\smallskip

As $T_0\in E(\E)^s\subseteq\delta(\E)=\Gamma\subseteq \Delta$ by \cite[Lemma~7.22]{normal}, it follows that $G:=N_\L(T_0)$ is a group with $N:=N_\N(T_0)\unlhd G$. Indeed, since $(\N,\Gamma,T)$ is a regular locality, $N$ is of characteristic $p$. Hence, by Lemma~\ref{L:CharpCharacterize}, $G$ is of characteristic $p$ if and only if $N_G(T)$ is of characteristic $p$. By \cite[Lemma~11.12]{Henke:Regular}, every automorphism of $\N$ leaves $E(\N)$ invariant, so $N_\L(T)$ acts by Lemma~\ref{L:KernelBasic}(b) on $E(\N)$ via conjugation. In particular, $N_\L(T)\subseteq G$ and thus $N_\L(T)=N_G(T)$. This shows the equivalence of (iii) and (iv) as required. 
\end{proof}

\section{Products}\label{S:Products}

In this section we prove Theorem~\ref{T:ProductMain} and Corollary~\ref{C:ProductsFusionSystems}. Indeed we state and prove here some more detailed results. The following theorem implies Theorem~\ref{T:ProductMain}. The reader might want to recall the notation introduced in Remark~\ref{R:ER}.

\begin{theorem}\label{T:ProductDetails}
Let $(\L, \Delta, S)$ be a regular locality and let $\N\norm \L$ be a partial normal subgroup of $\L$. Set 
\[T^* := F^*(\L) \cap S, \; T_0:=E(\N)\cap S,\; T:=S\cap\N\mbox{ and }\E:=\F_T(\N).\]
Let $H\leq N_\L(T^*)$ and set $\tH:=N_\N(T^*)H$. Fix $S_0\in \Syl_p(\tH)$ with $T\leq S_0$. Set 
\[\Delta_0 := \{ P \leq S_0 \mid P \cap \N \in \delta(\E)\}.\]
Then the following hold:
\begin{itemize}
\item[(a)] $\N H=H\N$ is a partial subgroup of $\L$. 
\item[(b)] $N_\N(T_0)=N_\N(T^*)$ and $\tH=N_{\N H}(T^*)=N_{\N H}(T_0)=N_\N(T_0)H$.
\item[(c)] $(\N H, \Delta_0, S_0)$ is a cr-complete locality with kernel $(\N,\delta(\E),T)$.
\item[(d)] Set $\tDelta_0:=\{P\leq S_0\colon PO_p(\N H)\in \Delta_0\}$. Then the following are equivalent: 
\begin{itemize}
\item $\tH$ is of characteristic $p$;
\item $(\N H,\Delta_0,S_0)$ is a linking locality;
\item $(\N H,\tDelta_0,S_0)$ is a regular locality.
\end{itemize}
\item[(e)] $\F_0:=\F_{S_0}(\N H)$ is saturated. Moreover, $\E$ and $E(\E)$ are normal subsystems of $\F_0$, $N_{\F_0}(T_0)=\F_{S_0}(\tH)$ and  
\[\F_0=\<(E(\E)S_0)_{\F_0},\F_{S_0}(\tH)\>.\]  
\end{itemize}
\end{theorem}

\begin{proof}
Set $\F:= \F_S(\L)$. As usual, we write $\Pi\colon\D\rightarrow\L$ for the product on $\L$.

\smallskip

\textbf{(a)} By definition of a regular locality, for every $P\leq S$, we have $P\in\Delta$ if and only if $P\cap T^*\in\Delta$. Thus,  Lemma~\ref{L:LocalitiesProp}(e) yields that
\begin{equation}\label{eq0}
 \mbox{ for all $u\in\W(\L)$, we have $u\in\D$ if and only if $S_u\cap T^*\in\Delta$.}
\end{equation}
In particular, $S_f \cap T^* \in \Delta$ for all $f\in\L$. If $n \in \N$ and $g\in \N_\L(T^*)$, then we can conclude that $(g^{-1},g,n)\in\D$ via $S_n\cap T^*$. Hence, by the axioms of a partial group, $gn$ is defined and $g^{-1}(gn)=\Pi(g^{-1},g,n)=n$. Note also that $S_{gn}\cap T^*\leq S_{(g,g^{-1},gn)}$. Using Lemma~\ref{L:LocalitiesProp}(f) we conclude that    
\[ S_{(g,n)} \cap T^* \leq S_{gn} \cap T^* \leq S_{(g,g^{-1},gn)} \cap T^* \leq S_{(g,g^{-1}(gn))} \cap T^* \leq S_{(g,n)} \cap T^*\]
and so 
\begin{equation}\label{eq1} 
S_{(g,n)} \cap T^* = S_{gn} \cap T^*\mbox{ for all $g\in N_\L(T^*)$ and all $n\in\N$.}
\end{equation}
Let $n\in \N$ and $g \in H$. We argue first that $\N H=H\N$. Observe that $u_1:=(g,g^{-1},n,g)\in\D$ via $S_n\cap T^*$ and $u_2:=(g,n,g^{-1},g)\in\D$ via $(S_n\cap T^*)^{g^{-1}}$. Hence, $ng=\Pi(u_1)=g(n^g)\in H\N$ and $gn=\Pi(u_2)=(n^{g^{-1}})g\in\N H$. This shows $\N H=H\N$. In particular, if $f=ng\in \N H$, then \cite[Lemma~1.4(f)]{loc1} yields $(g^{-1}, n^{-1}) \in \D$ and $f^{-1} = g^{-1}n^{-1}\in H\N=\N H$.

\smallskip

It remains thus to show that $\N H=H\N$ is closed under the partial product $\Pi\colon \D\rightarrow \L$ on $\L$. 
Let $w = (f_1,\dots,f_k) \in \D \cap \W(H\N)$. Then for every $1 \leq i \leq k $ we have $f_i = g_in_i$ for some $g_i \in H$ and $n_i \in \N$. Set $u = (g_1,n_1, \dots, g_k,n_k)$. Then by (\ref{eq1}) we get  $S_u\cap T^* = S_w\cap T^*$. Now \eqref{eq0} yields first $S_u\cap T^* = S_w\cap T^*\in\Delta$ and then $u \in \D$. By \cite[Lemma~3.4]{loc1}, $\Pi(u)=\Pi(g_1,\dots,g_k,n)$ for some $n\in \N$. Hence, by the axioms of a partial group, 
\[ \Pi(w) = \Pi(u) =\Pi(g_1,\dots,g_k,n)= \Pi(g_1,\dots,g_k)n \in H\N=\N H.\]
Thus $\N H$ is closed under the partial product and thus a partial subgroup. Hence, (a) holds.

\smallskip

\textbf{(b)} Lemma~\ref{L:T0R}(a) gives $N_\L(T_0)\leq N_\L(T^*)$ and $N_\N(T_0)=N_\N(T^*)$. It follows from the Dedekind Lemma \cite[Lemma~1.10]{loc1} that $N_{\N H}(T^*)=\N H\cap N_{\N H}(T^*)=N_\N(T^*)H=\tH$. As $H\leq N_\L(T^*)\subseteq N_\L(T_0)$ and $N_\L(T_0)$ is by \cite[Lemma~2.12(a)]{loc1} a partial subgroup of $\L$, the Dedekind Lemma yields similarly $N_{\N H}(T_0)=\N H\cap N_{\N H}(T_0)=N_\N(T_0)H$. This implies (b).

\smallskip

\textbf{(c,d)} Set $\Gamma:=\delta(\E)$ and $R:=E(C_\L(\N))\cap S$. We show first that 
\begin{equation}\label{eq2}
\D \cap \W(\N H) =\D_\Gamma\cap \W(\N H) 
\end{equation}
As $R\subseteq C_\L(\N)$, \cite[Lemma~3.5]{Henke:Regular} yields $\N\subseteq C_\L(R)\subseteq N_\L(R)$. Moreover, by Lemma~\ref{L:T0R}(a), $H\subseteq N_\L(T^*)\subseteq N_\L(R)$. As $N_\L(R)$ is by \cite[Lemma~2.12(a)]{loc1} a partial subgroup of $\L$, it follows that $\N H\subseteq N_\L(R)$. We use moreover that $O_p(\L)=O_p(\F)$ by \cite[Proposition~5]{subcentric} and Lemma~\ref{L:OpL}. So fixing $w=(f_1,\dots,f_n)\in\W(\N H)$, we have $RO_p(\F)=RO_p(\L)\leq S_w$. Hence, Lemma~\ref{L:deltaFR}(b) gives
\[S_w\in \delta(\F)=\Delta\Longleftrightarrow S_w\cap T\in\Gamma.\]
Notice that $S_w\in\D=\D_\Delta$ if and only if $S_w\in\Delta$. By \cite[Lemma~10.4]{Henke:Regular}, $\Gamma$ is overgroup-closed in $T$. Hence, $w\in\D_\Gamma$ implies $S_w\cap T\in\Gamma$. Moreover, by Lemma~\ref{L:T0R}(d), $\Gamma$ is closed under passing to $\L$-conjugates in $S$.  
Therefore, $S_w\cap T\in\Gamma$ implies $w\in\D_\Gamma$. So $S_w\cap T\in\Gamma$ if and only if $w\in\D_\Gamma$. This proves \eqref{eq2}. We show next:
\begin{equation}\label{eq4}
 S_0\mbox{ is maximal with respect to inclusion among the $p$-subgroups of }\N H.
\end{equation}
Let $S_1$ be a $p$-subgroup of $\N H$ such that $S_0\leq S_1$. Notice that $T\leq S_0\leq S_1$ and so $T_0=E(\N)\cap S\leq E(\N)\cap S_1$. By \cite[Lemma~11.13]{Henke:Regular}, $E(\N)\unlhd\L$. In particular, \cite[Lemma~3.1(c)]{loc1} yields that $T_0$ is a maximal $p$-subgroup of $E(\N)$. Hence, $T_0=E(\N)\cap S_1\unlhd S_1$. Now (b) yields $S_1\leq N_{\N H}(T_0)=\tH$. As $S_0\in \Syl_p(\tH)$, it follows that $S_1=S_0$. This proves \eqref{eq4}. 

\smallskip

Recall that by Lemma~\ref{L:T0R}(d), $\Gamma$ is closed under passing to $\L$-conjugates in $S$. As $T$ is by \cite[Lemma~3.1(c)]{loc1} a maximal $p$-subgroup of $\N$, we have $T=S_0\cap \N$. Since the elements of $\Gamma$ are contained in $T$ and $\N\unlhd \N H$, it follows that $\Gamma$ is closed under passing to $\N H$-conjugates in $S_0$. So by \eqref{eq2} and \eqref{eq4}, the hypothesis of Lemma~\ref{L:ConstructLocality} is fulfilled with $(\N H,S_0)$ in place of $(\L,S)$. Hence, it follows from this lemma that $(\N H,\Delta_0,S_0)$ is a locality with kernel $(\N,\Gamma,T)$. By \cite[Theorem~10.16(a)]{Henke:Regular}, $(\N,\Gamma,T)$ is a regular locality and in particular it is cr-complete. Hence, $(\N H,\Delta_0,S_0)$ is cr-complete by Theorem~\ref{T:crComplete}, i.e. (c) holds. Part (d) follows from (b) and Lemma~\ref{L:KernelRegularNLT0}.

\smallskip

\textbf{(e)} By \cite[Corollary~11.10]{Henke:Regular}, we have $T_0\in\delta(\E)=\Gamma\subseteq\Delta_0$. Hence, parts (b) and (c) together with Lemma~\ref{L:NDPNHP}(b) yield that $N_{\F_0}(T_0):=\F_{S_0}(N_{\N H}(T_0))=\F_{S_0}(\tH)$. Moreover, by Theorem~\ref{T:crComplete} and Proposition~\ref{P:crCompleteSaturated}(c), part (c) implies that $\F_0$ is saturated and $\E\unlhd\F_0$. Hence, $E(\E)\unlhd\F_0$ by \cite[Lemma~7.13(a)]{normal}. By Lemma~\ref{L:T0R}(b), we have $T_0=E(\E)\cap S$. Hence, (e) follows from Lemma~\ref{L:FusionFrattini}.
\end{proof}

If the subgroup $H$ in Theorem~\ref{T:ProductDetails} has the property that $S\cap H\in\Syl_p(H)$, then the following lemma says that we can choose the $p$-subgroup $S_0$ as a subgroup of $S$.

\begin{lemma}\label{P:ChooseS0}
Let $(\L,\Delta,S)$ be a regular locality and set $T^*:=F^*(\L)\cap S$. Let $\N\unlhd\L$, $T:=S\cap\N$ and $H\leq N_\L(T^*)$ such that $S\cap H\in\Syl_p(H)$. Then
\[S_0:=T(S\cap H)=S\cap (\N H)\in\Syl_p(N_\N(T^*)H).\]
\end{lemma}

\begin{proof}
Note that $N_\L(T^*)$ is a group, $S$ is a Sylow $p$-subgroup of $N_\L(T^*)$, $N_\N(T^*) \norm N_\L(T^*)$ and $H \leq N_\L(T^*)$ by assumption. As $S\cap H\in \Syl_p(H)$, Lemma \ref{lem:sylow.prod} yields that 
\[S_0:=(S \cap N_\N(T^*))(S \cap H)=T(S\cap H)\]
is a Sylow $p$-subgroup of $N_\N(T^*)H$. In particular, it is a consequence of Theorem~\ref{T:ProductDetails}(a),(c) that $S_0$ is a maximal $p$-subgroup of the partial group $\N H$ and thus equals $S\cap (\N H)$. 
\end{proof}

Assuming the hypothesis of Corollary~\ref{C:Products0} below, we have $\E\unlhd\F$ by \cite[Theorem~A]{normal}. Hence, it is a consequence of \cite[Lemma~7.13(a)]{normal} that $E(\F)\unlhd\F$ and thus $(E(\E)S_0)_\F$ is well-defined. This is implicitly used in the statement of part (b).

\begin{corollary}\label{C:Products0}
 Let $(\L,\Delta,S)$ be a regular locality, $T^*:=F^*(\L)\cap S$, $\N\unlhd\L$, and $H\leq N_\L(T^*)$ with $S\cap H\in \Syl_p(H)$. Set $T:=S\cap T$, $\E:=\F_T(\N)$, $S_0:=T(S\cap H)$, $\tH:=N_\N(T^*)H$ and 
\[\E H:=\F_{S_0}(\N H).\]
Then the following hold:
\begin{itemize}
 \item [(a)] $\E H$ is a saturated subsystem of $\F$, $\E\unlhd\E H$ and $E(\E)\unlhd \E H$.
 \item [(b)] $S_0\in\Syl_p(\tH)$ and $\E H=\<E(\E)S_0,\F_{S_0}(\tH)\>$ where $E(\E)S_0:=(E(\E)S_0)_\F=(E(\E)S_0)_{\E H}$.
 \item [(c)] If $\mD$ is a saturated subsystem of $\F$ such that $E(\E)\unlhd\mD$ and $\F_{S_0}(\tH)\subseteq\mD$, then $\E H\subseteq \mD$.
\end{itemize}
\end{corollary}

\begin{proof}
By Lemma~\ref{P:ChooseS0}, $S_0\in\Syl_p(\tH)$. Therefore, setting $\Delta_0:=\{P\leq S_0\colon P\cap T\in\delta(\E)\}$, the hypothesis and thus the conclusion of Theorem~\ref{T:ProductDetails} holds. Theorem~\ref{T:ProductDetails}(e) implies now that (a) holds and $\E H=\<(E(\E)S_0)_{\E H},\F_{S_0}(\tH)\>$.  By \cite[Remark~2.27]{normal}, $(E(\E)S_0)_{\E H}=(E(\E)S_0)_\F$. Hence (b) holds as well.

\smallskip

Let now $\mD$ be as in (c) and suppose $\mD$ is a subsystem over $R\leq S$. As $\F_{S_0}(\tH)\subseteq\mD$, we have in particular that $S_0\leq R$. Using \cite[Remark~2.27]{normal} again, we observe that $(E(\E)S_0)_\F=(E(\E)S_0)_\mD\subseteq\mD$. As $\F_{S_0}(\tH)\subseteq\mD$ by assumption, it follows now from (b) that $\E H\subseteq\mD$ and so (c) holds.
\end{proof}

\begin{proof}[Proof of Corollary~\ref{C:ProductsFusionSystems}]
By \cite[Lemma~10.4]{Henke:Regular}, there exists a regular locality $(\L,\Delta,S)$ over $\F$. Moreover, by \cite[Theorem~A]{normal}, there exists $\N\unlhd\L$ with $T=S\cap\N$ and $\E=\F_T(\N)$. By \cite[Theorem~E(d)]{normal}, we have $T^*:=F^*(\F)\cap S=F^*(\L)\cap S$. Moreover, Lemma~\ref{L:T0R}(b) gives  $T_0:=E(\E)\cap S=E(\N)\cap S$. In particular, by \cite[Lemma~10.4, Remark~10.12]{Henke:Regular}, we have $T^*\in\delta(\F)=\Delta\subseteq \F^s$. So $N_\F(T^*)$ is constrained and $N_\L(T^*)$ is by Lemma~\ref{L:NDPNHP}(b) a model for $N_\F(T^*)$. Using Lemma~\ref{L:T0R}(a), one observes that $N_\N(T_0)=N_\N(T^*)\unlhd N_\L(T^*)$. Moreover, Lemma~\ref{L:T0R}(c) gives $N_\E(T_0)=\F_{T}(N_\N(T^*))$. Hence, it follows from \cite[Proposition~I.6.2]{AKO} that $N_\E(T_0)\unlhd N_\F(T^*)$.

\smallskip

Let now $G$ be an arbitrary model for $N_\F(T^*)$. As $N_\E(T_0)\unlhd N_\F(T^*)$, it follows from \cite[Theorem~II.7.5]{AKO} that there exists a unique normal subgroup $N$ of $G$ with $S\cap N=T$ and $\F_T(N)=N_\E(T_0)$. Let now $H\leq G$ with $S\cap H\in\Syl_p(H)$. Set $S_0:=T(S\cap H)$. By \cite[Theorem~III.5.10]{AKO}, there exists an isomorphism $\alpha\colon G\rightarrow N_\L(T^*)$ with $\alpha|_S=\id_S$. Notice that $N\alpha\unlhd N_\L(T^*)$ and $H\alpha\leq N_\L(T^*)$ with 
\[T=S\cap (N\alpha),\;\F_T(N\alpha)=\F_{T}(N)=N_\E(T^*),\]
\[S\cap H=S\cap (H\alpha)\in\Syl_p(H\alpha)\mbox{ and }\F_{S_0}(N H)=\F_{S_0}((NH)\alpha).\] 
By \cite[Theorem~II.7.5]{AKO}, $N_\N(T^*)$ is the unique normal subgroup of $N_\L(T^*)$ realizing $N_\E(T^*)$. Hence, $N\alpha=N_\N(T^*)$ and $\F_{S_0}(NH)=\F_{S_0}(N_\N(T^*)(H\alpha))$. Therefore the assertion follows from Corollary~\ref{C:Products0} applied with $H\alpha$ in place of $H$. 
\end{proof}

\appendix

\section{Normal pairs of transporter systems}\label{A}

It is not within the scope of this paper to construct extensions of  localities, but Theorem~\ref{T:FInvariantOliverExtensions2} and our theorems on kernels provide some tools for examining the properties of existing extensions. The results in this appendix may help to compare our theorems on kernels to theorems on extensions of linking systems in the literature. Such extensions  have been studied in various places  starting with \cite{BCGLOExtensions}. A more transparent algebraic framework is used in  \cite{OliverExtensions, AOV1}, where the definition of a normal pair of linking systems is crucial (see also \cite[Definition~III.4.12]{AKO}). This definition naturally generalizes to a definition of a ``normal pair of transporter systems'', which we state below. Transporter systems were defined by Oliver and Ventura \cite{OliverVentura} generalizing the concept of a linking system. The goal of this appendix is to show that localities with kernels correspond to normal pairs of transporter systems (cf. Proposition~\ref{L:KernelsToNormalPairs} and Theorem~\ref{T:NormalPairsToKernels}). We use here a correspondence between localities and transporter systems which was observed by Chermak \cite[Appendix~A]{Ch}.

\smallskip

For any functor $\alpha\colon \C\rightarrow \m{D}$ between categories and any objects $P,Q\in\C$, we let
\[\alpha_{P,Q}\colon \Mor_\C(P,Q)\rightarrow \Mor_\m{D}(\alpha(P),\alpha(Q))\] denote the induced map between morphisms sets. Moreover, we set $\alpha_P:=\alpha_{P,P}$.

\smallskip

The literature on linking systems and transporter systems is mostly written in ``left-hand notation''. Therefore, in this appendix (unlike in the rest of the paper) we will write maps on the left hand side of the argument and conjugate from the left. In particular, if $G$ is a finite group and $\Delta$ is a set of subgroups of $G$, then $\T_\Delta(G)$ denotes the category whose object set is $\Delta$ and such that the morphism set $\Mor_{\T_\Delta(G)}(P,Q)$ between any two objects $P,Q\in\Delta$ is the set of all $g\in G$ with
\[^g\!P:=P^{g^{-1}}=gPg^{-1}\leq Q.\]
A \emph{transporter system} associated to a fusion system $\F$ over $S$ is a category $\T$ whose object set is an $\F$-closed collection of subgroups of $S$, together with functors
\[\T_{\Ob(\T)}(S)\xrightarrow{\;\;\;\delta\;\;\;}\T\xrightarrow{\;\;\;\pi\;\;\;}\F\]
subject to certain axioms. In particular, $\delta$ is the identity on objects and $\pi$ is the inclusion on objects, $\delta$ is injective on morphism sets, $\pi$ is surjective on morphism sets, and $\pi\circ\delta$ sends an element $g\in \Mor_{\T_{\Ob(\T)}(S)}(P,Q)$ to the corresponding conjugation map $c_{g^{-1}}\colon P\rightarrow Q,x\mapsto {}^g\!x$. The reader is referred to \cite[Definition~3.1]{OliverVentura} for the precise definition. If $P,Q\in\Ob(\T)$ and $\phi\in\Mor_\T(P,Q)$, we will usually write $\pi(\phi)$ for $\pi_{P,Q}(\phi)$.  The following definition is non-standard.

\begin{definition}
If $(\T,\delta,\pi)$ is a transporter system associated to a fusion system $\F$, then we will say that $(\T,\delta,\pi)$ is a transporter system \emph{over} $\F$ if $\F$ is $\Ob(\T)$-generated, i.e.
\[\F=\<\pi(\phi)\colon P,Q\in\Ob(\T),\;\phi\in\Mor_\T(P,Q)\>.\]
\end{definition}

If $(\T,\delta,\pi)$ is a linking system associated to a saturated fusion system $\F$, then it follows from Alperin's Fusion Theorem (cf. \cite[Theorem~I.3.6]{AKO}) that $(\T,\delta,\pi)$ is a transporter system \emph{over} $\F$. However, in general there can be transporter systems associated to a fusion system $\F$ which are not transporter systems over $\F$ (but just transporter systems over a subsystem of $\F$).

\smallskip

For a locality $(\L,\Delta,S)$, there is a transporter system $\T_\Delta(\L)$ over $\F_S(\L)$ defined. The object set of the category $\T_\Delta(\L)$ is $\Delta$, and for $P,Q\in\Delta$, the set $\Mor_{\T_\Delta(\L)}(P,Q)$ consists of all $f\in\L$ with
\[P\leq S_{f^{-1}}\mbox{ and }^f\!P:=P^{f^{-1}}\leq Q.\]
It turns out (see \cite[Proposition~A.3]{Ch}) that $(\T_\Delta(\L),\delta,\pi)$ is a transporter system, where the functor $\delta\colon \T_\Delta(S)\rightarrow \T_\Delta(\L)$ is the identity on objects and the inclusion on morphism sets, and the functor $\pi\colon \T_\Delta(\L)\rightarrow \F_S(\L)$ is the inclusion on objects and  sends a morphism $f\in \Mor_{\T_\Delta(\L)}(P,Q)$ to $c_f|_P\in\Hom_{\F_S(\L)}(P,Q)$.

\smallskip

To review some notation and basic results, let $(\T,\delta,\pi)$ be a transporter system and fix $P,Q\in\Ob(\T)$. If $P\leq Q$, then the morphism
\[\iota_{P,Q}:=\delta_{P,Q}(1)\]
is regarded as an ``inclusion map''. If $\phi\in\Mor_\T(P,Q)$ and $P'\leq P$, then set
\[\phi(P'):=\pi(\phi)(P').\]
By \cite[Lemma~A.6]{OliverVentura}, a morphism $\phi\in\Mor_\T(P,Q)$ is an isomorphism in the categorical sense if and only if $\pi(\phi)$ is an isomorphism, i.e. if and only if $\phi(P)=Q$. It is shown in \cite[Lemma~3.2(c)]{OliverVentura} that for all $P',Q'\in\Delta$ with $P'\leq P$, $Q'\leq Q$ and all $\phi\in\Mor_\T(P,Q)$ with $\phi(P')\leq Q'$, there exists a unique morphism $\phi|_{P',Q'}\in\Mor_\T(P',Q')$ with
\[\phi\circ \iota_{P',P}=\iota_{Q',Q}\circ \phi|_{P',Q'}.\]
The morphism $\phi|_{P',Q'}$ is called the \emph{restriction} of $\phi$ to a morphism from $P'$ to $Q'$.

\begin{lemma}\label{L:Restrictions}
Let $(\T,\delta,\pi)$ be a transporter system associated to some fusion system over a $p$-group $S$. Then the following hold:
\begin{itemize}
 \item [(a)] Let $P''\leq P'\leq P$ and $Q''\leq Q'\leq Q$ be objects in $\T$. Let $\phi\in\Mor_\T(P,Q)$ with $\phi(P')\leq Q'$ and $\phi(P'')\leq Q''$. Then
 \[\phi|_{P'',Q''}=(\phi|_{P',Q'})|_{P'',Q''}.\]
 \item [(b)] Let $P,Q,R,P',R'\in\Ob(\T)$, $\phi\in\Mor_\T(P,Q)$ and $\psi\in\Mor_\T(Q,R)$ with $P'\leq P$ and $(\psi\circ\phi)(P')\leq R'\leq R$. Then
 \[(\psi\circ \phi)|_{P',R'}=\psi|_{\phi(P'),R'}\circ \phi|_{P',\phi(P')}.\]
 \item [(c)] Let $x\in S$ and $\alpha=\delta_S(x)$. For every $P\in\Delta$ we have $\alpha(P)={}^x\!P\in\Delta$ and
 \[\alpha|_{P,\alpha(P)}=\delta_{P,\alpha(P)}(x).\]
 \item [(d)] Let $P,Q,P'\in\Ob(\T)$ and $\psi,\phi\in\Mor_\T(P,Q)$ such that $P'\leq P$, $Q':=\psi(P')=\phi(P')$ and $\phi|_{P',Q'}=\psi|_{P',Q'}$. Then $\phi=\psi$.
\end{itemize}
\end{lemma}

\begin{proof}
\textbf{(a)} Setting $\phi'=\phi|_{P',Q'}$ it follows from the definition of the restriction that
\[\iota_{Q',Q}\circ \phi'=\phi\circ\iota_{P',P}\mbox{ and }\iota_{Q'',Q'}\circ \phi'|_{P'',Q''}=\phi'\circ \iota_{P'',P'}.\]
Hence,
\begin{eqnarray*}
\iota_{Q'',Q}\circ \phi'|_{P'',Q''}
 &=&\iota_{Q',Q}\circ \iota_{Q'',Q'}\circ \phi'|_{P'',Q''}\\
 &=&\iota_{Q',Q}\circ \phi'\circ \iota_{P'',P'}\\
 &=&\phi\circ \iota_{P',P}\circ \iota_{P'',P'}\\
 &=&\phi\circ \iota_{P'',P}.
\end{eqnarray*}
This implies (a).

\smallskip

\textbf{(b)} See \cite[Lemma~A.7(b)]{Ch}.

\smallskip

\textbf{(c)} By axiom (B) of a transporter system as stated in \cite[Definition~3.1]{OliverVentura}, we have $\pi(\alpha)=c_{x^{-1}}$ and thus $^x\!P=P^{x^{-1}}=\alpha(P)$. As $\delta$ is a functor, it follows that
\[\iota_{\alpha(P),S}\circ \delta_{P,\alpha(P)}(x)=\delta_{P,S}(1\cdot x)=\delta_{P,S}(x\cdot 1)=\delta_S(x)\circ \delta_{P,S}(1)=\alpha\circ \iota_{P,S}.\]
This shows (c).

\textbf{(d)} Assume the hypothesis of (d) and set $\phi_0:=\phi|_{P',Q'}=\psi|_{P',Q'}$. It follows from the definition restrictions that
\[\phi\circ\iota_{P',P}=\iota_{Q',Q}\circ \phi_0=\psi\circ\iota_{P',P}.\]
As every morphism in $\T$ is by \cite[Lemma~3.2(d)]{OliverVentura} an epimorphism, it follows that $\phi=\psi$ as required.
\end{proof}

Central to the considerations in this appendix is the following definition which generalizes the definition of a normal pair of linking systems (cf. \cite[Definition~III.4.12]{AKO}).

\begin{definition}\label{D:NormalPair}
Fix a pair of fusion systems $\F_0\subseteq\F$ over $p$-groups $S_0\leq S$ such that $\F_0$ is $\F$-invariant. Let $\T_0\subseteq \T$ be transporter systems over $\F_0\subseteq\F$ respectively. Then $\T_0$ is called \emph{normal} in $\T$ (written as $\T_0\unlhd\T$) if the following hold:
\begin{itemize}
 \item [(i)] $\Ob(\T)=\{P\leq S\colon P\cap S_0\in\Ob(\T_0)\}$;
 \item [(ii)] for all $P\in\Ob(\T_0)$ and $\psi\in\Mor_\T(P,S_0)$, there are morphisms $\gamma\in\Aut_\T(S_0)$ and $\psi_0\in\Mor_{\T_0}(P,S_0)$ such that $\psi=\gamma\circ \psi_0$; and
 \item [(iii)] for all $\gamma\in\Aut_\T(S_0)$, $P,Q\in\Ob(\T_0)$, and $\psi\in\Mor_{\T_0}(P,Q)$,
 \[\gamma|_{Q,\gamma(Q)}\circ \psi\circ (\gamma|_{P,\gamma(P)})^{-1}\in\Mor_{\T_0}(\gamma(P),\gamma(Q)).\]
\end{itemize}
We say then also that $\T_0\unlhd\T$ is a \emph{normal pair of transporter systems} over $\F_0\subseteq\F$.
\end{definition}

Saying that $\T_0\subseteq \T$ are transporter systems over $\F_0\subseteq\F$ respectively means here more precisely that $\T_0$ is a subcategory of $\T$, that $(\T,\delta,\pi)$ is a transporter system over $\F$ (for appropriate functors $\delta$ and $\pi$) and that $(\T_0,\delta|_{\T_{\Ob(\T_0)}(S_0)},\pi|_{\T_0})$ is a transporter system over $\F_0$.

\smallskip

We prove now first that kernels of localities lead to normal pairs of transporter systems.

\begin{prop}\label{L:KernelsToNormalPairs}
Let $(\N,\Gamma,T)$ be a kernel of a locality $(\L,\Delta,S)$. Then $\T_\Gamma(\N)\unlhd \T_\Delta(\L)$ is a normal pair of transporter systems over $\F_T(\N)\subseteq\F_S(\L)$.
\end{prop}

\begin{proof}
It follows from the discussion above that $\T_\Gamma(\N)\subseteq\T_\Delta(\L)$ are transporter systems over $\F_T(\N)\subseteq\F_S(\L)$. By Lemma~\ref{L:KernelBasic}(d), $\F_T(\N)$ is $\F_S(\L)$-invariant. Referring to the properties (i),(ii),(iii) in Definition~\ref{D:NormalPair}, it follows from the definition of a kernel that (i) holds. Property (ii) is a consequence of Lemma~\ref{L:FrattiniSplitting}. Property (iii) holds since $\N\unlhd\L$.
\end{proof}

We outline now how a locality can be constructed from a transporter system and use this afterwards to show that normal pairs of transporter systems lead to localities with kernels.

\smallskip

Let $(\T,\delta,\pi)$ be a transporter system over a fusion system $\F$ on $S$. Write $\Iso(\T)$ for the set of all isomorphisms in $\T$. Define a relation $\uparrow_\T$ on $\Iso(\T)$ by writing
\[\phi\uparrow_\T \psi\]
if $\psi$ restricts to $\phi$. More precisely this means that $\phi\uparrow_\T\psi$, if $\phi\in\Iso_\T(P',Q')$ and $\psi\in\Iso_\T(P,Q)$ for some $P,P',Q,Q'\in\Ob(\T)$ with $P'\leq P$, $Q'\leq Q$, $\psi(P')=Q'$ and $\phi=\psi|_{P',Q'}$. Let then $\equiv_\T$ be the smallest equivalence relation on $\Iso(\T)$ containing $\uparrow_\T$. Write $[\phi]$ for the $\equiv_\T$-equivalence class of $\phi\in\Iso(\T)$ and $\L(\T)$ for the set of all equivalence classes of $\Iso(\T)$.

\begin{remark}\label{R:equiv}
We have $\phi\equiv_\T\psi$ if and only if there exists a sequence $\phi_1,\phi_2,\cdots,\phi_k\in \Iso(\T)$ such that, for all $i=1,2,\dots,k-1$,
\[\phi_i\uparrow_\T\phi_{i+1}\mbox{ or }\phi_{i+1}\uparrow_\T\phi_i.\]
\end{remark}

By $\D$ denote the set of tuples $w=(f_1,f_2,\dots,f_k)\in\W(\L(\T))$ for which there exist $\phi_i\in f_i$ for $i=1,\dots,k$ such that the composition $\phi_1\circ\phi_2\circ\cdots\circ \phi_k$ is defined in the category $\T$. Moreover, given such $w$ and $\phi_i$, set
\[\Pi(w):=[\phi_1\circ\phi_2\circ\cdots\circ \phi_k].\]
The map $\Pi\colon \D\longrightarrow\L(\T)$ is well-defined. Together with $\Pi$ and the map
\[\L(\T)\longrightarrow\L(\T),[\phi]\mapsto[\phi^{-1}],\]
which is also well-defined, the set $\L(\T)$ forms a partial group by \cite[Proposition~A.9]{Ch}.  Moreover, the map
\[S\longrightarrow\L(\T),\;x\mapsto [\delta_S(x)]\]
is an injective homomorphism of partial groups, and its image is a subgroup of $\L(\T)$. We will usually identify $x\in S$ with $[\delta_S(x)]\in \L(\T)$. With this identification, $S$ is a subgroup of $\L(\T)$. Setting $\Delta:=\Ob(\T)$, it is shown in  \cite[Proposition~A.13]{Ch} that $(\L(\T),\Delta,S)$ is a locality. As we assume that $(\T,\delta,\pi)$ is a transporter system over $\F$, it follows from \cite[Lemma~4.4(a)]{Henke:2020} that $(\L(\T),\Delta,S)$ is a locality over $\F$. We will use these properties throughout without further reference.

\begin{remark}\label{R:ChermakConstruction}
The partial group structure on $\L(\T)$ constructed above is not exactly the same as the one constructed by Chermak \cite[Appendix~A]{Ch}. The reason is that Chermak uses consistently the ``right-hand notation'' for maps and also for the category $\T$. So if $\Pi\colon\D\rightarrow \L(\T)$ is the partial product defined above and $\Pi'\colon \D'\rightarrow \L(\T)$ is the product constructed by Chermak, then \[\D'=\{(f_n,f_{n-1},\dots,f_1)\colon (f_1,f_2,\dots,f_n)\in\D\}\]
and $\Pi'(f_n,f_{n-1},\dots,f_1)=\Pi(f_1,f_2,\dots,f_n)$ for all $(f_1,f_2,\dots,f_n)\in\D$. In particular, conjugation by $f\in\L(\T)$ with respect to the partial group with product $\Pi'$ corresponds to conjugation by $f^{-1}$ in the partial group constructed above.
\end{remark}

The main goal of this appendix is to prove the following theorem.

\begin{theorem}\label{T:NormalPairsToKernels}
Let $\F_0\subseteq\F$ be fusion systems over $S_0\leq S$ such that $\F_0$ is $\F$-invariant. Let $\T_0\unlhd\T$ be a normal pair of transporter systems over $\F_0\subseteq\F$. Set $\Delta:=\Ob(\T)$, $\Gamma:=\Ob(\T_0)$, $\L:=\L(\T)$ and
\[\N:=\{f\in\L\colon f\cap \Iso(\T_0)\neq\emptyset\}.\]
Then $(\N,\Gamma,S_0)$ is a kernel of $(\L,\Delta,S)$ with $\F_{S_0}(\N)=\F_0$. Moreover, there is an isomorphism $\N\rightarrow\L(\T_0)$, which is the identity on $S_0$, and there exists an invertible functor $\T_\Gamma(\N)\rightarrow \T_0$ which is the identity on $\Gamma$.
\end{theorem}

\begin{remark}\label{R:KernelsToNormalPairs}
Note that Theorem~\ref{T:NormalPairsToKernels} allows us to conclude results about normal pairs of transporter systems from our results about kernels of localities. Let us point out one example: Call a transporter system $(\T,\delta,\pi)$ over $\F$ cr-complete if $\F^{cr}\subseteq\Ob(\T)$. Consider a normal pair of transporter systems $\T_0\unlhd\T$ over $\F_0\subseteq\F$. Then it follows from Theorem~\ref{T:crComplete} that $\T_0$ is cr-complete if and only if $\T$ is cr-complete. Moreover, if so, then $\F_0$ is normal in $\F$. As a particular consequence, if $\T_0\unlhd\T$ is a normal pair of linking systems over $\F_0\subseteq\F$ (as defined in \cite{AOV1} and \cite{OliverExtensions}), then $\F_0$ is always normal in $\F$.
\end{remark}

To prove  Theorem~\ref{T:NormalPairsToKernels}, we assume the following hypothesis:

\smallskip

\textbf{From now on let $\F_0\subseteq\F$ be fusion systems over $S_0\leq S$ respectively such that $\F_0$ is $\F$-invariant, and let $\T_0\unlhd\T$ be a normal pair of transporter systems over $\F_0\subseteq\F$ respectively.}

\smallskip

\textbf{More precisely, let $(\T,\delta,\pi)$ be a transporter system over $\F$ and $(\T,\delta|_{\T_{\Ob(\T_0)}(S)},\pi|_{\T_0})$ be a transporter system over $\F_0$.}

\begin{lemma}\label{L:FrattiniTransporter}
Let $P,Q\in\Ob(\T_0)$ and $\psi\in\Iso_\T(P,Q)$. Then there exist $\gamma\in\Aut_\T(S_0)$, $R\in\Ob(\T_0)$ and $\psi_0\in\Iso_{\T_0}(P,R)$ such that $\gamma(R)=Q$ and $\psi=\gamma|_{R,Q}\circ \psi_0$.
\end{lemma}

\begin{proof}
Notice that $\hat{\psi}:=\iota_{Q,S_0}\circ \psi\in \Mor_{\T_0}(P,S_0)$ with $\hat{\psi}|_{P,Q}=\psi$. Hence, by axiom (ii) of Definition~\ref{D:NormalPair}, there exists $\gamma\in\Aut_\T(S_0)$ and $\hat{\psi}_0\in \Mor_{\T_0}(P,S_0)$ such that
\[\hat{\psi}=\gamma\circ \hat{\psi}_0.\]
Setting $R:=\hat{\psi}_0(P)$ and $\psi_0:=\hat{\psi}_0|_{P,R}\in\Iso_{\T_0}(P,R)$, we have $\gamma(R)=Q$ and it follows from Lemma~\ref{L:Restrictions}(b) that $\psi=\hat{\psi}|_{P,Q}=\gamma|_{R,Q}\circ \psi_0$.
\end{proof}

\begin{lemma}\label{L:RestrictToT0}
Let $P,\ov{P},Q,\ov{Q}\in\Ob(\T_0)$ with $P\leq \ov{P}$ and $Q\leq \ov{Q}$. Let $\psi\in\Iso_\T(\ov{P},\ov{Q})$ with $\psi(P)\leq Q$ and $\psi|_{P,Q}\in\Mor_{\T_0}(P,Q)$. Then $\psi$ is a morphism in $\T_0$.
\end{lemma}

\begin{proof}
As $\ov{P}$ is a $p$-group, $P$ is subnormal in $\ov{P}$. As $\Ob(\T_0)$ is overgroup-closed in $S_0$, induction on $|\ov{P}:P|$ allows us to reduce to the case that $P\unlhd\ov{P}$. Moreover, replacing $Q$ by $\psi(P)$, we may assume that $\psi(P)=Q$, $\phi:=\psi|_{P,Q}\in\Iso_{\T_0}(P,Q)$ and $Q=\psi(P)\unlhd\psi(\ov{P})=\ov{Q}$. It follows from \cite[Lemma~3.3]{OliverVentura} that $\phi\circ\delta_P(x)=\delta_Q(\pi(\psi)(x))\circ \phi$ for all $x\in \ov{P}$. As $\pi(\psi)(\ov{P})\leq \ov{Q}$, this implies that $\phi\circ \delta_P(\ov{P})\circ\phi^{-1}\leq \delta_Q(\ov{Q})$. Hence, by Axiom II of a transporter systems (as stated in \cite[Definition~3.1]{OliverVentura}) applied to the transporter system $\T_0$, there exists $\ov{\phi}\in\Mor_{\T_0}(\ov{P},\ov{Q})$ with $\ov{\phi}|_{P,Q}=\phi=\psi|_{P,Q}$. 
It follows now from Lemma~\ref{L:Restrictions}(d) that $\psi=\ov{\phi}$ is a morphism in $\T_0$.
\end{proof}

Set now $\Delta:=\Ob(\T)$, $\Gamma:=\Ob(\T_0)$, $\L:=\L(\T)$ and
\[\N:=\{f\in\L\colon f\cap \Iso(\T_0)\neq\emptyset\}.\]
As before we write $\Pi\colon\D\rightarrow\L$ for the product on $\L$.

\begin{notation}
For $P\leq S$ set $P_0:=P\cap S_0$. Similarly define $Q_0,P'_0,Q'_0$ for subgroups $Q,P',Q'\leq S$. For $P,Q\in\Delta$ and $\phi\in\Iso_\T(P,Q)$ set moreover
\[\phi^0:=\phi|_{P_0,Q_0}.\]
\end{notation}

As $S_0$ is strongly closed, we have in the situation above that $\phi(P_0)=\pi(\phi)(P_0)=(\pi(\phi)(P))\cap S_0=Q_0$. So $\phi^0$ is well-defined and an element of $\Iso_\T(P_0,Q_0)$. Observe also that $\phi^0\uparrow_\T\phi$ and thus $[\phi]=[\phi^0]$ for all $\phi\in\Iso(\T)$.

\begin{lemma}\label{L:phi0uparrow}
 If $\phi,\psi\in\Iso(\T)$ with $\psi\uparrow_\T\phi$, then $\psi^0\uparrow_\T\phi^0$.
\end{lemma}

\begin{proof}
 Let $P'\leq P$ and $Q'\leq Q$ be objects in $\T$ such that $\phi\in\Mor_\T(P,Q)$, $\psi\in\Mor_\T(P',Q')$, $\phi(P')\leq Q'$ and $\phi|_{P',Q'}=\psi$. Applying Lemma~\ref{L:Restrictions}(a) twice gives then
 \[\phi^0|_{P'_0,Q'_0}=\phi|_{P'_0,Q'_0}=\psi|_{P_0',Q_0'}=\psi^0.\]
 Hence, $\psi^0\uparrow_\T\phi^0$.
\end{proof}

\begin{lemma}\label{L:equiv}
The following hold:
\begin{itemize}
\item [(a)] If $\phi,\psi\in\Iso(\T_0)$, then
$\phi\equiv_\T\psi$ if and only if $\phi\equiv_{\T_0}\psi$.
\item [(b)] Let $\phi\in\Iso(\T_0)$ and $\psi\in\Iso_\T(P,Q)$ for some $P,Q\in\Ob(\T_0)$. If $\phi\equiv_\T\psi$, then $\psi\in\Iso(\T_0)$.
\item [(c)] If $\phi\in\Iso(\T)$ with $[\phi]\in\N$, then $\phi^0\in\Iso(\T_0)$.
\end{itemize}
\end{lemma}

\begin{proof}
\textbf{(a,b)} Clearly the relation $\uparrow_{\T_0}$ is contained in the relation $\uparrow_\T$ and so, if $\phi,\psi\in\Iso(\T_0)$ with $\phi\equiv_{\T_0}\psi$, then $\phi\equiv_\T\psi$.

\smallskip

Suppose now that $\phi\in\Iso(\T_0)$ and $\psi\in\Iso_\T(P,Q)$ for some $P,Q\in\Ob(\T_0)$. Assume $\phi\equiv_\T\psi$. To show (a) and (b) it remains to show that $\psi\in\Iso(\T_0)$ and $\phi\equiv_{\T_0}\psi$. By Remark~\ref{R:equiv}, there exists a series $\phi=\phi_1,\phi_2,\dots,\phi_n=\psi\in\Iso(\T)$ such that, for all $i=1,2,\dots,n-1$, we have $\phi_i\uparrow_\T \phi_{i+1}$ or $\phi_{i+1}\uparrow_\T\phi_i$. Notice that $\phi^0=\phi$ and $\psi^0=\psi$ as $P,Q\in\Ob(\T_0)$. Thus,
Lemma~\ref{L:phi0uparrow} allows us to replace $\phi_1,\phi_2,\dots,\phi_n$ by $\phi_1^0,\phi_2^0,\dots,\phi_n^0$. Thus, we may assume that $\phi_1,\phi_2,\dots,\phi_n$ are isomorphisms in $\T$ between objects of $\T_0$. Then Lemma~\ref{L:RestrictToT0} implies that, for all $i=1,2,\dots,n-1$, $\phi_i\in\Iso(\T_0)$ if and only if $\phi_{i+1}\in\Iso(\T_0)$. As $\phi_1=\phi\in\Iso(\T_0)$ by assumption, it follows therefore inductively that $\phi_1,\phi_2,\dots,\phi_n\in\Iso(\T_0)$. In particular, $\psi=\phi_n\in\Iso(\T_0)$. For $\alpha,\beta\in\Iso(\T_0)$ it is easy to observe that $\alpha\uparrow_\T\beta$ if and only if $\alpha\uparrow_{\T_0}\beta$. Hence, we have also $\phi_i\uparrow_{\T_0} \phi_{i+1}$ or $\phi_{i+1}\uparrow_{\T_0}\phi_i$ for $i=1,2,\dots,n-1$. This shows $\phi\equiv_{\T_0}\psi$. So (a) and (b) hold.

\smallskip

\textbf{(c)} If $\phi\in\Iso(\T)$ with $[\phi]\in\N$, then by definition of $\N$, there exists $\psi\in\Iso(\T_0)$ with $\psi\in[\phi]$. Then $\phi^0\equiv_\T\phi\equiv_\T\psi$. Hence, it follows from part (b) that $\phi^0\in\Iso(\T_0)$.
\end{proof}

For the following lemma recall that we identify $x\in S$ with $[\delta_S(x)]\in\L$.

\begin{lemma}\label{L:NPartialSubgroup}
\begin{itemize}
\item [(a)] $\N$ is a partial subgroup of $\L$ with $\N\cap S=S_0$.
\item [(b)] Set $\L_0:=\L(\T_0)$, write $[\phi]_0$ for the $\equiv_{\T_0}$-equivalence class of $\phi\in\Iso(\T_0)$, and identify $x\in S_0$ with $[\delta_{S_0}(x)]_0$. Then the map
\[\theta\colon\N\rightarrow \L_0,f\mapsto f\cap \Iso(\T_0)\]
is an isomorphism of localities which restricts to the identity on $S_0$.
\end{itemize}
\end{lemma}

\begin{proof}
We prove first:
\begin{equation}\label{E:Ninversion}
 \N\mbox{ is closed under inversion.}
\end{equation}
For the proof note that every $n\in\N$ can be written as $n=[\phi]$ for some $\phi\in\Iso(\T_0)$. Then $\phi^{-1}\in \Iso(\T_0)$ and it follows from the definition of the inversion on $\L$ that $n^{-1}=[\phi^{-1}]\in\N$.  This shows \eqref{E:Ninversion}. We argue next that
\begin{equation}\label{E:thetaBijection}
\mbox{the map $\theta$ is well-defined and a bijection.}
\end{equation}
Indeed, it follows from Lemma~\ref{L:equiv}(a) that $[\phi]\cap \Iso(\T_0)=[\phi]_0$ for all $\phi\in\Iso(\T_0)$. Hence, $\theta$ is well-defined and surjective. Note also that $f\cap g\neq\emptyset$ for all $f,g\in\N$ with $\theta(f)=\theta(g)$. As the equivalence classes of $\equiv_\T$ form a partition of $\Iso(\T)$, we can therefore conclude that the map $\theta$ is  injective. So \eqref{E:thetaBijection} holds. We show next:
\begin{equation}\label{E:S0inN}
 S_0\subseteq \N\cap S\mbox{ and }\theta|_{S_0}=\id_{S_0}.
\end{equation}
For the proof let $x\in S_0$. Lemma~\ref{L:Restrictions}(c) implies $\delta_S(x)|_{S_0}=\delta_{S_0}(x)\in\Aut_{\T_0}(S_0)\subseteq\Iso(\T_0)$ and so $x=[\delta_S(x)]=[\delta_{S_0}(x)]\in \N$. Moreover, $x=[\delta_S(x)]$ gets mapped $[\delta_{S_0}(x)]_0$, which we also identify with $x$ (as stated in part (b)).  This proves \eqref{E:S0inN}. We show next:
\begin{equation}\label{E:AHom}
\mbox{$\N$ is a partial subgroup and $\theta$ is a homomorphism of partial groups.}
\end{equation}
For the proof let $w:=(n_1,n_2,\dots,n_k)\in\W(\N)\cap \D$. As $w\in\D$, there exist $\phi_i\in n_i$ for $i=1,2,\dots,k$ such that the composition $\phi_1\circ\phi_2\circ\cdots\circ \phi_k$ is defined in $\T$. Note that then also the composition $\phi_1^0\circ\phi_2^\circ \cdots \circ \phi_k^0$ is defined. Moreover, as $\phi_i^0\uparrow_\T\phi_i$, we have $\phi_i^0\in n_i$ for $i=1,2,\dots,k$. By Lemma~\ref{L:equiv}(c), $\phi_i^0\in\Iso(\T_0)$. So replacing $\phi_1,\phi_2,\dots,\phi_k$ by $\phi_1^0,\phi_2^0,\dots,\phi_k^0$, we may assume $\phi_i\in\Iso(\T_0)$. Then $\phi_1\circ\phi_2\circ \cdots \circ \phi_k\in\Iso(\T_0)$ and hence
\[\Pi(w)=[\phi_1\circ \phi_2\circ\cdots\circ \phi_k]\in\N.\]
Together with \eqref{E:Ninversion} this shows that $\N$ is a partial subgroup.

\smallskip

Note also that $\theta(n_i)=n_i\cap\Iso(\T_0)=[\phi_i]_0$ for $i=1,2,\dots,k$. Moreover, the composition $\phi_1\circ\phi_2\circ\cdots\circ\phi_k$ is defined in $\T_0$. Hence, writing $\Pi_0\colon\D_0\rightarrow\L_0$ for the product on $\L_0$ (defined in the usual way), it follows that $\theta^*(w)=(\theta(n_1),\theta(n_2),\dots,\theta(n_k))\in \D_0$ and
\[\Pi_0(\theta^*(w))=[\phi_1\circ\phi_2\circ\cdots\circ\phi_n]_0=[\phi_1\circ\phi_2\circ\cdots\circ\phi_n]\cap \Iso(\T_0)=\Pi(w)\cap\Iso(\T_0)=\theta(\Pi(w)).\]
This proves that $\theta$ is a homomorphism of partial groups and \eqref{E:AHom} holds.

\smallskip

For the proofs of the next two properties recall Definition~\ref{D:Homomorphisms} and the notation introduced there. We show next that
\begin{equation}\label{E:thetaIso}
 \mbox{$\theta$ is an isomorphism of partial groups.}
\end{equation}
Writing again $\Pi_0\colon\D_0\rightarrow\L_0$ for the product on $\L_0$, it is by \eqref{E:thetaBijection} and \eqref{E:AHom}  sufficient to prove that $\D_0\subseteq \theta^*(\D\cap\W(\N))$. For the proof of this property let $u=(g_1,g_2,\dots,g_n)\in\D_0$ and fix $\phi_i\in g_i$ for $i=1,2,\dots,n$ such that the composition $\phi_1\circ\phi_2\circ\cdots\circ\phi_n$ is defined in $\T_0$. Note that such $\phi_i$ exist by definition of $\D_0$. Then $\phi_i\in\Iso(\T_0)$ and hence $f_i:=[\phi_i]\in\N$ with $\theta(f_i)=g_i$. As the composition $\phi_1\circ\phi_2\circ\cdots\circ\phi_n$ is defined in $\T_0$, it is also defined in $\T$. Therefore, $w:=(f_1,\dots,f_n)\in\D\cap \W(\N)$ with $\theta^*(w)=u$. This shows $\D_0\subseteq\theta^*(\D\cap\W(\N))$ and thus \eqref{E:thetaIso}.

\smallskip

Note that (b) holds by \eqref{E:S0inN} and \eqref{E:thetaIso}. Moreover, $S_0\subseteq S\cap \N$. Recall also that $\N$ is a partial subgroup of $\L$ by \eqref{E:AHom} and that $(\L_0,S_0,\Gamma)$ is by \cite[Proposition~A.13]{Ch} a locality. In particular, $S_0$ is a maximal $p$-subgroup of $\L_0$ and thus by (b) also a maximal $p$-subgroup of $\N$. This implies $S\cap \N=S_0$.
\end{proof}

\begin{lemma}\label{L:7}
 Let $P,Q,R,P',Q',R'\in\Ob(\T_0)$, $\phi\in\Iso_\T(P,Q)$, $\psi\in\Iso_\T(P',Q')$, $\phi_0\in\Iso_{\T_0}(P,R)$, $\psi_0\in\Iso_{\T_0}(P',R')$, $\alpha,\gamma\in \Aut_\T(S_0)$ such that $\alpha(R)=Q$, $\gamma(R')=Q'$,
 \[\phi=\alpha|_{R,Q}\circ\phi_0\mbox{ and }\psi=\gamma|_{R',Q'}\circ\psi_0.\]
If $\phi\equiv_\T\psi$, then $\gamma^{-1}\circ \alpha\in \Aut_{\T_0}(S_0)$.
\end{lemma}

\begin{proof}
 By Remark~\ref{R:equiv}, there exists a series $\phi=\beta_1,\beta_2,\dots,\beta_n=\psi\in\Iso(\T)$ such that $\beta_i\uparrow\beta_{i+1}$ or $\beta_{i+1}\uparrow \beta_i$ for $i=1,\dots,n-1$. By Lemma~\ref{L:FrattiniTransporter}, every $\beta_i$ can be factored as a restriction of an automorphism $\alpha_i\in\Aut_\T(S_0)$ with an isomorphism in $\T_0$. We may assume that $\alpha_1=\alpha$ and $\alpha_n=\gamma$. If $\alpha_{i+1}^{-1}\circ \alpha_i\in\Aut_{\T_0}(S_0)$ for $i=1,\dots,n-1$, then
 \[\gamma^{-1}\circ\alpha=\alpha_n^{-1}\circ\alpha_1=(\alpha_n^{-1}\circ \alpha_{n-1})\circ (\alpha_{n-1}^{-1}\circ \alpha_{n-2})\circ\cdots \circ (\alpha_2^{-1}\circ\alpha_1)\in\Aut_{\T_0}(S_0).\]
Thus, we may reduce to the case $\phi\uparrow_\T\psi$ or $\psi\uparrow_\T\phi$. As $\gamma^{-1}\circ\alpha\in\Aut_{\T_0}(S_0)$ if and only if $\alpha^{-1}\circ\gamma=(\gamma^{-1}\circ\alpha)^{-1}\in\Aut_{\T_0}(S_0)$, we may indeed assume that
\[\phi\uparrow_\T\psi.\]
This means that $P\leq P'$, $\psi(P)=Q\leq Q'$ and $\phi=\psi|_{P,Q}$. Setting $X:=\psi_0(P)$, Lemma~\ref{L:Restrictions}(b) gives
\[\alpha|_{R,Q}\circ\phi_0=\phi=\psi|_{P,Q}=\gamma|_{X,Q}\circ \psi_0|_{P,X}.\]
As $\phi_0$ and $\psi_0$ are morphisms in $\T_0$, this implies together with Lemma~\ref{L:Restrictions}(b) that
\[(\gamma^{-1}\circ \alpha)|_{R,X}=(\gamma|_{X,Q})^{-1}\circ \alpha|_{R,Q}=\psi_0|_{P,X}\circ \phi_0^{-1}\in\Iso_{\T_0}(R,X).\]
Hence, Lemma~\ref{L:RestrictToT0} gives $\gamma^{-1}\circ\alpha\in\Aut_{\T_0}(S_0)$.
\end{proof}

\begin{lemma}\label{L:NPartialNormal}
We have $\N\unlhd\L$.
\end{lemma}

\begin{proof}
Recall that $\N$ is a partial subgroup of $\L$ by Lemma~\ref{L:NPartialSubgroup}(a). Let $f\in\L$ and $n\in\N$ such that $(f^{-1},n,f)\in\D$. Then there exist
\[\phi^{-1}\in f^{-1},\;\rho\in n,\;\psi\in f\]
such that the composition $\phi^{-1}\circ\rho\circ\psi$ is defined. Replacing $\phi,\rho,\psi$ by $\phi^0,\rho^0,\psi^0$, we may assume that, for some $P,Q,P',Q'\in\Ob(\T_0)$, we have
\[\psi\in\Iso_\T(Q',P'),\;\rho\in\Iso_{\T_0}(P',P),\;\phi^{-1}\in\Iso_\T(P,Q).\]
By Lemma~\ref{L:FrattiniTransporter}, there exist then $R,R'\in\Ob(\T_0)$, $\phi_0\in\Mor_{\T_0}(P,R)$, $\psi_0\in\Mor_{\T_0}(P',R')$ and $\alpha,\gamma\in\Aut_{\T_0}(S_0)$ such that $\alpha(R)=Q$, $\gamma(R')=Q'$,
\[\phi^{-1}=\alpha|_{R,Q}\circ\phi_0\mbox{ and }\psi^{-1}=\gamma|_{R',Q'}\circ\psi_0.\]
As $f=[\psi]$, we have $[\psi^{-1}]=f^{-1}=[\phi^{-1}]$ and so $\psi^{-1}\equiv_\T\phi^{-1}$. Thus, Lemma~\ref{L:7} (applied with $(\psi^{-1},\phi^{-1})$ in place of $(\phi,\psi)$) implies that $\gamma^{-1}\circ\alpha\in\Aut_{\T_0}(S_0)$. By definition of the product $\Pi$ on $\L$, we have
\[n^f=\Pi(f^{-1},n,f)=[\phi^{-1}\circ\rho\circ\psi].\]
Using Lemma~\ref{L:Restrictions}(b), observe that
\begin{eqnarray*}
 \phi^{-1}\circ\rho\circ\psi &=& \alpha|_{R,Q}\circ\phi_0\circ\rho\circ\psi_0^{-1}\circ (\gamma|_{R',Q'})^{-1}\\
 &=& \alpha|_{R,Q}\circ \phi_0\circ\rho\circ\psi_0^{-1}\circ (\gamma^{-1}\circ \alpha)|_{\alpha^{-1}(Q'),R'}\circ (\alpha|_{\alpha^{-1}(Q'),Q'})^{-1}.
\end{eqnarray*}
As $\phi_0\circ\rho\circ\psi_0^{-1}\circ (\gamma^{-1}\circ \alpha)|_{\alpha^{-1}(Q'),R'}\in\Iso(\T_0)$, it follows from axiom (iii) in Definition~\ref{D:NormalPair} that $\phi^{-1}\circ\rho\circ\psi\in\Iso(\T_0)$. Hence, $n^f=[\phi^{-1}\circ\rho\circ\psi]\in\N$. This shows the assertion.
\end{proof}

\begin{proof}[Proof of
Theorem~\ref{T:NormalPairsToKernels}]
By Lemmas~\ref{L:NPartialSubgroup}(a) and \ref{L:NPartialNormal}, $\N$ is a partial normal subgroup of $\L$ with $S\cap \N=S_0$. Using that $\Ob(\T_0)\subseteq\Ob(\T)$ and axiom (i) in Definition~\ref{D:NormalPair} holds, we can see that $\Gamma=\{P\cap S_0\colon P\in\Delta\}\subseteq\Delta$. Hence, $(\N,\Gamma,S_0)$ is a kernel of $(\L,\Delta,S)$.

\smallskip

By Lemma~\ref{L:NPartialSubgroup}(b), there is an isomorphism $\theta\colon \N\rightarrow\L_0:=\L(\T_0)$ which (with appropriate identifications) restricts to the identity on $S_0$. In particular, $\F_{S_0}(\N)=\F_{S_0}(\L_0)$ and there exists an invertible functor $\T_\Gamma(\N)\rightarrow \T_\Gamma(\L_0)$ which is the identity on $\Gamma$. We have seen in our general discussion above that $\F=\F_S(\L)$. So we have similarly that $\F_0=\F_{S_0}(\L_0)=\F_S(\N)$. It is moreover shown in \cite[Lemma~A.15]{Ch} that there is an invertible functor $\T_\Gamma(\L_0)\rightarrow \T_0$ which restricts to the identity on $\Gamma$. This implies the assertion.
\end{proof}

\bibliographystyle{alpha}
\bibliography{my.books}

\newcommand{\etalchar}[1]{$^{#1}$}
\begin{thebibliography}{BCG{\etalchar{+}}07}

\bibitem[AKO11]{AKO}
M.~Aschbacher, R.~Kessar, and B.~Oliver.
\newblock {\em Fusion systems in algebra and topology}, volume 391 of {\em
  London Mathematical Society Lecture Note Series}.
\newblock Cambridge University Press, Cambridge, 2011.

\bibitem[AOV12]{AOV1}
K.~K.~S. Andersen, B.~Oliver, and J.~Ventura.
\newblock Reduced, tame and exotic fusion systems.
\newblock {\em Proc. Lond. Math. Soc. (3)}, 105(1):87--152, 2012.

\bibitem[Asc11]{AschbacherGeneralized}
M.~Aschbacher.
\newblock The generalized {F}itting subsystem of a fusion system.
\newblock {\em Mem. Amer. Math. Soc.}, 209(986):vi+110, 2011.

\bibitem[Asc18]{AschbacherFSCT}
M.~Aschbacher.
\newblock On fusion systems of component type.
\newblock {\em Mem. Amer. Math. Soc.}, 257(1236):vi+182, 2018.

\bibitem[BCG{\etalchar{+}}05]{controlling}
C.~Broto, N.~Castellana, J.~Grodal, R.~Levi, and B.~Oliver.
\newblock Subgroup families controlling {$p$}-local finite groups.
\newblock {\em Proc. London Math. Soc. (3)}, 91(2):325--354, 2005.

\bibitem[BCG{\etalchar{+}}07]{BCGLOExtensions}
C.~Broto, N.~Castellana, J.~Grodal, R.~Levi, and B.~Oliver.
\newblock Extensions of {$p$}-local finite groups.
\newblock {\em Trans. Amer. Math. Soc.}, 359(8):3791--3858, 2007.

\bibitem[CH21]{normal}
A.~Chermak and E.~Henke.
\newblock Fusion systems and localities -- a dictionary.
\newblock {\em arXiv:1706.05343v3}, 2021.

\bibitem[Che13]{Ch}
A.~Chermak.
\newblock Fusion systems and localities.
\newblock {\em Acta Math. 211}, pages 47--139, 2013.

\bibitem[Che16]{loc3}
A.~Chermak.
\newblock {Finite Localities III}.
\newblock {\em arXiv:1610.06161v2}, 2016.

\bibitem[Che22]{loc1}
A.~Chermak.
\newblock {Finite Localities I}.
\newblock {\em {Forum Math. Sigma}}, 10:Paper No. e43, 31, 2022.

\bibitem[GL16]{Glauberman/Lynd}
G.~Glauberman and J.~Lynd.
\newblock {Control of fixed points and existence and uniqueness of centric
  linking systems}.
\newblock {\em Invent. Math.}, pages 1--44, 2016.
\newblock Published online first: March 19 2016, doi:10.1007/s00222-016-0657-5.

\bibitem[Gon15]{Gonzalez}
Alex Gonzalez.
\newblock An extension theory for partial groups and localities.
\newblock {\em arXiv:1507.04392v2}, 2015.

\bibitem[Hen13]{Henke.products}
E.~Henke.
\newblock Products in fusion systems.
\newblock {\em J. Algebra}, 376:300--319, 2013.

\bibitem[Hen15]{Henke:2015a}
E.~Henke.
\newblock Products of partial normal subgroups.
\newblock {\em Pacific J. Math.}, 279(1-2):255--268, 2015.

\bibitem[Hen18]{centralizer.normal}
E.~Henke.
\newblock Centralizers of normal subsystems revisited.
\newblock {\em J. Algebra}, 511:364--387, 2018.

\bibitem[Hen19]{subcentric}
E.~Henke.
\newblock Subcentric linking systems.
\newblock {\em Trans. Amer. Math. Soc.}, 371:3325--3373, 2019.

\bibitem[Hen21a]{Henke:Regular}
E.~Henke.
\newblock {Commuting partial normal subgroups and regular localities}.
\newblock {\em arXiv:2103.00955v2}, 2021.

\bibitem[Hen21b]{Henke:2020}
E.~Henke.
\newblock {Extensions of homomorphisms between localities}.
\newblock {\em Forum Math. Sigma}, 9:Paper No. e63, 2021.

\bibitem[Hen21c]{Henke:Subnormal}
E.~Henke.
\newblock {Normalizers and centralizers of subnormal subsystems of fusion
  systems}.
\newblock {\em arXiv:2109.14038}, 2021.

\bibitem[Hen21d]{Henke:NK}
E.~Henke.
\newblock {Some products in fusion systems and localities}.
\newblock {\em arXiv:2107.00301}, 2021.

\bibitem[KS04]{KS}
H.~Kurzweil and B.~Stellmacher.
\newblock {\em The theory of finite groups}.
\newblock Universitext. Springer-Verlag, New York, 2004.
\newblock An introduction, Translated from the 1998 German original.

\bibitem[MS12]{MS}
U.~Meierfrankenfeld and B.~Stellmacher.
\newblock Applications of the {FF}-{M}odule {T}heorem and related results.
\newblock {\em J. Algebra}, 351:64--106, 2012.

\bibitem[Oli10]{OliverExtensions}
B.~Oliver.
\newblock Extensions of linking systems and fusion systems.
\newblock {\em Trans. Amer. Math. Soc.}, 362(10):5483--5500, 2010.

\bibitem[Oli13]{Oliver:2013}
B.~Oliver.
\newblock Existence and uniqueness of linking systems: {C}hermak's proof via
  obstruction theory.
\newblock {\em Acta Math.}, 211(1):141--175, 2013.

\bibitem[OV07]{OliverVentura}
B.~Oliver and J.~Ventura.
\newblock Extensions of linking systems with {$p$}-group kernel.
\newblock {\em Math. Ann.}, 338(4):983--1043, 2007.

\bibitem[RS09]{RS}
K.~Roberts and S.~Shpectorov.
\newblock On the definition of saturated fusion systems.
\newblock {\em J. Group Theory}, 12(5):679--687, 2009.

\end{thebibliography}
\end{document}